\documentclass[reqno,12pt,letterpaper]{amsart}
\usepackage{amsmath,amssymb,amsthm,graphicx,mathrsfs,url,bbm,array,enumerate}
\usepackage[usenames,dvipsnames]{xcolor}
\usepackage[colorlinks=true,linkcolor=Red,citecolor=Green]{hyperref}
\usepackage{mathabx}
\usepackage{mathtools}

\usepackage[
    backend=biber,
    style=alphabetic,
    giveninits=true,
    maxalphanames=10,
    maxnames=10,
    doi=false,
    url=false,
    isbn=false,
]{biblatex}
\DeclareFieldFormat{pages}{#1}
\renewbibmacro{in:}{%
  \ifentrytype{article}
    {}
    {\bibstring{in}%
     \printunit{\intitlepunct}}}
\DeclareFieldFormat
  [article,inbook,incollection,inproceedings,patent,thesis,unpublished]
  {title}{\mkbibemph{#1}}
\DeclareFieldFormat{journaltitle}{#1\isdot}
\DeclareFieldFormat[article]{volume}{\mkbibbold{#1}}
\DeclareFieldFormat[article]{number}{\bibstring{number}\addnbspace #1}

\renewbibmacro*{journal+issuetitle}{%
  \usebibmacro{journal}%
  \setunit*{\addspace}%
  \iffieldundef{series}
    {}
    {\newunit
     \printfield{series}%
     \setunit{\addspace}}%
  \printfield{volume}%
  \setunit{\addspace}%
  \usebibmacro{issue+date}%
  \setunit{\addcomma\space}%
  \printfield{number}%
  \setunit{\addcolon\space}%
  \usebibmacro{issue}%
  \setunit{\addcomma\space}%
  \printfield{eid}
  \newunit}

 \DeclareFieldFormat*{title}{#1}
 \DeclareFieldFormat{title}{#1}
 \DeclareFieldFormat
   [article,inbook,incollection,inproceedings,patent,thesis,unpublished]
   {title}{\mkbibquote{#1\isdot}}
 \DeclareFieldFormat
   [suppbook,suppcollection,suppperiodical]
   {title}{#1}

\newcolumntype{L}{>{$}l<{$}}

\def\?[#1]{\textbf{[#1]}\marginpar{\Large{\textbf{??}}}}
\newtheorem{thm}{Theorem}
\newtheorem*{thm*}{Theorem}
\newtheorem{prop}{Proposition}

\newtheorem{defi}[prop]{Definition}
\newtheorem{lem}[prop]{Lemma}
\newtheorem{cor}[prop]{Corollary}
\newtheorem{rem}[prop]{Remark}

\numberwithin{equation}{section}
\numberwithin{prop}{section}

\theoremstyle{definition}

\DeclareMathOperator{\supp}{supp}

\DeclareMathOperator{\WF}{WF}

\newcommand{\bbH}{\mathbb H}

\newcommand{\bbR}{\mathbb R}

\newcommand{\bbZ}{\mathbb Z}

\definecolor{green}{rgb}{0,0.8,0}

\title[Resolvent estimate without evenness assumption]{Resolvent estimate for asymptotically hyperbolic surfaces without evenness assumption}
\author{Wenzheng Wang}
\date{July 2026}
\address{Mathematical Institute,
Rheinische Friedrich-Wilhelms-Universität Bonn,
Endenicher Allee 60,
53115 Bonn,
Germany}
\email{wwenzheng819@gmail.com}

\begin{document}
\maketitle

\begin{abstract}
    We prove a sharp cut-off resolvent estimate for the Laplacian on non-compact Riemannian manifolds whose ends satisfy the Cardoso--Vodev assumption and whose geometry near the trapped set agrees with that of an even asymptotically hyperbolic manifold with the corresponding resolvent estimate. As applications, we show the global Strichartz and spectral projection estimates on negatively curved asymptotically hyperbolic surfaces, without assuming the evenness at infinity.
\end{abstract}


\section{Introduction}
Let $(M,g)$ be a complete Riemannian manifold of dimension $n\geq 2$. Let $\Delta_g$ be the Laplace--Beltrami operator on $M$. Let 
\[
R_M(\lambda\pm i0)=(-\Delta_g - \frac{(n-1)^2}{4}-  (\lambda\pm i0)^2)^{-1},\quad \lambda\in \bbR_+.
\]
be the resolvent of the Laplacian at $\lambda^2$, where $\pm i0$ means taking the limit from the upper half-space or the lower half-space. Let $\chi\in C_c^\infty(M)$ be a cutoff function on $M$. We will study the behavior of $\chi R_M(\lambda\pm i0)\chi$ for $\lambda$ large enough. When there is no trapping (see \eqref{eq:trapped-set}), one typically has
\begin{equation}
    \|{\chi R_M(\lambda\pm i0)\chi}\|_{L^2\to L^2}\leqslant C\lambda^{-1}, \quad \lambda\geqslant C.
\end{equation}
See, for example, \cite{chen2018resolventII} for the asymptotically hyperbolic case, and \cite{cardoso2002uniform} for more general assumptions on the ends of the manifold.

When $M$ has trapping, an additional logarithmic loss is in general unavoidable, even for surfaces; see \cite[Theorem 7.1, 7.2]{dyatlov2019mathematical}. For \emph{even} asymptotically hyperbolic surface $\tilde{M}$ with negative curvature, we have the sharp estimate (\cite{tao2021spectral})
\begin{equation}\label{eq:far-away}
    \|\chi R_{\tilde{M}}(\lambda\pm i0)\chi\|_{L^2\to L^2}\leqslant C\lambda^{-1}\log\lambda, \quad \lambda\geqslant C.
\end{equation}

In \cite{Bounded_HSTZ}, together with the resolvent estimate \eqref{eq:far-away}, one can establish the spectral projection estimates and global Strichartz estimates. The aim of this article is to extend the estimate \eqref{eq:far-away} to a more general family of surfaces. Namely,

\begin{thm}[Main result]
\label{Thm}
Let $(M,g)$ be a complete, non-compact, connected Riemannian manifold of dimension $n$ satisfying:
\begin{enumerate}
\item[\textup{(i)}] The ends of $M$ satisfy the Cardoso--Vodev assumption in Proposition~\ref{prop:cv}.
\item[\textup{(ii)}] There exists an even asymptotically hyperbolic manifold $(\tilde{M},\tilde{g})$ such that, on a neighborhood of the trapped set, $(M,g)$ and $(\tilde{M},\tilde{g})$ are isometric. In particular, two trapped sets are the same: $K=\tilde{K}$.
\item[\textup{(iii)}] The Laplace operator $\Delta_{\tilde{g}}$ on $\tilde{M}$ satisfies the estimate \eqref{eq:far-away}.
\end{enumerate}
Then for every $\chi\in C_c^\infty(M)$ and $\lambda$ large enough,
\begin{equation}\label{main:eq}
\|\chi(-\Delta_g-\frac{(n-1)^2}{4}-(\lambda\pm i0)^2)^{-1}\chi\|_{L^2(M)\to L^2(M)}\leqslant C\lambda^{-1}\log\lambda.
\end{equation}
\end{thm}

By Kato's smoothing principle(see for \cite[Theorem 7.2]{dyatlov2019mathematical}), \eqref{main:eq} implies the local smoothing estimate:
\begin{equation}\label{eq:locsm}
    \int_{-\infty}^{\infty} \|\chi e^{-it\Delta_g} \beta(\sqrt{-\Delta_g}/\lambda) f\|_{L^2}^2 dt \leq C\lambda^{-1}\log \lambda \|f\|_{L^2}^2,\quad \beta\in C_c^{\infty}((1/2,2)), \, \chi\in C_c^{\infty}(M).
\end{equation}

In particular, Theorem~\ref{Thm} applies to non-even asymptotically hyperbolic surfaces.
\begin{cor}\label{cor:main}
For an asymptotically hyperbolic surface (as in Definition~\ref{def:AH}) with negative curvature, the estimates \eqref{main:eq} and \eqref{eq:locsm} hold. 
\end{cor}

Using Corollary~\ref{cor:main}, we also generalize the results of \cite{Bounded_HSTZ} to asymptotically hyperbolic surfaces without the evenness assumption.
\begin{thm}\label{thm2}
    Suppose $M$ is an asymptotically hyperbolic surface with negative curvature. Then 
    \begin{enumerate}
        \item Assume $M$ has no resonance at the bottom of the continuous spectrum.
    Then we have the Strichartz estimate for $\frac{1}{p}+\frac{1}{q}=\frac{1}{2}, \,p,q\ge 2$ and $(p,q)\neq(2,+\infty)$, we have 
     \begin{equation*}
        \|e^{-it\Delta_g}u_0\|_{L^pL^q(\bbR\times M)}\leqslant C\| u_0\|_{L^2(M)}
    \end{equation*}
    for $u_0$ orthogonal to all the $L^2$ eigenfunctions.
    \item We have the spectral projection estimate for $\lambda\gg 1$ and $\delta\in(0,1)$:
    \begin{equation*}
        \|\mathbf{1}_{[\lambda,\lambda+\delta]}(\sqrt{-\Delta_g}) f\|_{L^q(M)}\leqslant C\lambda^{\mu(q)}\delta^{1/2}\|f\|_{L^2(M)},
    \end{equation*}
    where 
    \begin{equation}\label{mu}
        \mu(q)=\left\{\begin{aligned}
            &1/2-2/q,\quad q\geqslant 6,\\
            &\frac{1}{2}(1/2-1/q),\quad q\in (2,6].
        \end{aligned}
        \right.
    \end{equation}
    \end{enumerate}

\end{thm}

\subsection{Related work}
The resolvent estimate \eqref{main:eq} has a long history and is closely related to the spectral gap. We first review the development of resolvent estimates. One of the first general results is due to Nicolas Burq in \cite{burq1998decroissance}, where he proved an exponential high-energy bound for the cutoff resolvent in a rather general trapping setting. After this, Cardoso--Vodev \cite{cardoso2002uniform} followed this idea and proved more refined resolvent estimates for infinite volume Riemannian manifolds under additional assumptions on the geometry near infinity. These results do not require a specific hyperbolic structure of the trapped set.

We next recall what is meant by a spectral gap. Resonances are the poles of the meromorphic continuation of the resolvent $R_M(\lambda)$. For even asymptotically hyperbolic manifolds, the resolvent extends meromorphically to the whole complex plane; see \cite{mazzeo1987meromorphic,guillarmou2005meromorphic,vasy2013microlocal}. We say that there is an essential spectral gap if there exists $\beta>0$ such that there are only finitely many resonances in the strip
\[
-\beta<\operatorname{Im}\lambda<0.
\]
The spectral gap is closely connected to cutoff resolvent estimates on the real line. If one has a resonance-free strip together with a polynomial resolvent bound in this strip, then the semiclassical maximum principle of \cite[Lemma 4.7]{burq2004smoothing} gives an improved cutoff resolvent estimate on the real axis. Thus, in this sense, a spectral gap together with a polynomial bound contains stronger information than a real-axis resolvent estimate alone.

For hyperbolic trapping, Nonnenmacher--Zworski \cite{nonnenmacher2009quantum} proved an essential spectral gap and a resolvent estimate under a topological pressure condition. A geometrically different setting is normally hyperbolic trapping. Wunsch--Zworski \cite{wunsch2011resolvent} proved a spectral gap and resolvent estimates for semiclassical operators with normally hyperbolic trapped sets. In their setting the trapped set is smooth and symplectic. Later, Nonnenmacher--Zworski \cite{nonnenmacher2015decay} weakened the regularity assumptions on the incoming and outgoing tails $\Gamma_\pm$ and removed the codimension assumption. Their result also applies to contact Anosov flows and gives exponential decay of correlations from the corresponding spectral gap. Dyatlov \cite{dyatlov2016spectral} then obtained the optimal size of the spectral gap in the smooth normally hyperbolic setting of \cite{wunsch2011resolvent}, together with a polynomial resolvent bound. In \cite{dyatlov2015resonance}, under $r$-normal hyperbolicity and an additional pinching condition, Dyatlov also obtained a resonance-free strip.

Another important development is the fractal uncertainty principle. For convex co-compact hyperbolic surfaces, the topological pressure condition used in \cite{nonnenmacher2009quantum} does not apply once the Hausdorff dimension $\delta$ of the limit set is at least $1/2$. Dyatlov--Zahl \cite{dyatlov2016spectralfup} introduced the fractal uncertainty principle in their study of spectral gaps and resolvent estimates for convex co-compact hyperbolic manifolds. Bourgain--Dyatlov \cite{bourgain2018spectral} then removed the pressure condition in dimension two and proved that every convex co-compact hyperbolic surface has an essential spectral gap. Vacossin \cite{vacossin2024spectral,vacossin2023resolvent} generalized \cite{bourgain2018spectral} to quantum monodromy maps and scattering outside convex obstacles in dimension two.  Tao \cite{tao2021spectral} applied the result of \cite{vacossin2024spectral} to prove an essential spectral gap together with resolvent estimates for even negatively curved asymptotically hyperbolic surfaces, without assuming constant curvature. For recent advances in higher-dimensional fractal uncertainty principle, see \cite{backus2025fractal,cohen2D,cohen2025fractal}.

We now turn to Strichartz and spectral projection estimates. On the real hyperbolic spaces $\bbH^n$, Anker--Pierfelice \cite{anker2009nonlinear} and Ionescu--Staffilani \cite{ionescu2009semilinear} independently proved global-in-time Strichartz estimates and applied them to nonlinear Schrödinger equations. Bouclet \cite{bouclet2011strichartz} studied more general asymptotically hyperbolic manifolds. He constructed a microlocal parametrix near infinity and proved local-in-time Strichartz estimates without loss outside a compact set. In the nontrapping case, this also gives Strichartz estimates without loss on the whole manifold.

The first lossless Strichartz estimate in the presence of trapping is due to Burq--Guillarmou--Hassell \cite{burq2010strichartz}. They proved Strichartz estimates without loss in the presence of hyperbolic trapping under a pressure condition. Their proof combines logarithmic-time dispersive estimates with local smoothing. For convex co-compact hyperbolic surfaces without the pressure condition, Wang \cite{wang2019strichartz} proved a Strichartz estimate with $\epsilon$ loss of derivative. On the other hand, Chen--Hassell \cite{chen2018resolventII} constructed the high-energy resolvent and studied the spectral measure on nontrapping asymptotically hyperbolic manifolds, obtaining restriction and spectral multiplier estimates. Based on this spectral measure construction, Chen proved global-in-time Strichartz estimates without loss in \cite{chen2018resolvent}.

Spectral projection estimates were developed in parallel with these Strichartz estimates. Anker--Germain--L\'eger \cite{anker2023spectral} studied spectral projectors on geometrically finite hyperbolic surfaces of infinite volume. In the convex co-compact case, they obtained estimates which are optimal in the window size and frequency up to subpolynomial losses. Their proof uses the Bourgain--Dyatlov resolvent estimate together with improved Strichartz and smoothing estimates. For compact manifolds of nonpositive or negative curvature, Blair--Huang--Sogge \cite{blair2024improved,blair2024strichartz} and Huang--Sogge \cite{huang2025curvature,huang2025strichartz} obtained logarithmic improvements of spectral projection and Strichartz estimates. In particular, they proved lossless frequency-localized Strichartz estimates on time intervals of Ehrenfest size $\lambda^{-1}\log\lambda$.

Huang--Sogge--Tao--Zhang \cite{Bounded_HSTZ} extended these ideas to even negatively curved asymptotically hyperbolic surfaces. They use local smoothing and half-localized $L^2\to L^q$ resolvent estimates to give lossless Strichartz estimates and optimal spectral projection estimates on negatively curved even asymptotically hyperbolic surfaces, in particular on all convex co-compact hyperbolic surfaces, without the pressure condition. If there is no resonance at the bottom of the continuous spectrum, the Strichartz estimate is global in time after projecting away from the $L^2$ eigenfunctions. More recently, Zhang \cite{zhang2026strichartz} used a similar gluing and background-manifold argument for asymptotically conic surfaces. She obtained lossless unit-time Strichartz estimates and spectral projection estimates for surfaces with Euclidean ends when a sufficiently large neighborhood of the trapped set has negative curvature.














\subsection{Outline of the proof}
We will prove the estimate for the localized resolvent
\[
\chi_L R_M(\lambda\pm i0)\chi_R,\quad \chi_L,\chi_R\in C_c^{\infty}(M).
\]
Let \(K\) denote the trapped set defined in
\eqref{eq:trapped-set}, and let
\(\chi_0\in C_c^\infty(M)\) be a fixed cutoff satisfying
\(\chi_0=1\) in a neighborhood of \(\pi(K)\).
The argument is organized according to the positions of the supports of
\(\chi_L\) and \(\chi_R\) relative to the trapped region \(\pi(K)\).

\begin{itemize}
    \item
    Suppose first that both \(\chi_L\) and \(\chi_R\) are supported away
    from \(\pi(K)\). We claim that
    \begin{equation}\label{eq:outline-1}
         \|\chi_L R_M(\lambda\pm i0)\chi_R\|_{L^2\to L^2}\leq C\lambda^{-1}.
    \end{equation}
    By the propagation estimate in
    Proposition~\ref{prop:propagation}, there exists a cutoff
    \(\chi_\infty\), supported in the nontrapping region sufficiently far
    from \(\pi(K)\), such that
    \[
    \|\chi_L R_M(\lambda\pm i0)\chi_R\|_{L^2\to L^2}
    \lesssim
    \|\chi_\infty R_M(\lambda\pm i0)\chi_R\|_{L^2\to L^2}
    +\lambda^{-1}.
    \]
    Applying the same propagation estimate to the adjoint resolvent,
    using
    \[
    R_M(\lambda\pm i0)^*=R_M(\lambda\mp i0),
    \]
    reduces the problem further to estimating
    \[
    \|\chi_\infty
      R_M(\lambda\pm i0)
      \chi_{\infty}\|_{L^2\to L^2},
    \]
    where \(\chi_\infty\) is supported far away from $\pi(K)$. This term is controlled by the
    Cardoso--Vodev estimate stated in
    Proposition~\ref{prop:cv}.

    \item
    We next consider the case in which one cutoff is supported away from
    the trapped region while the other one may intersect it. By symmetry,
    it is enough to assume that
    \[
    \supp\chi_L\cap\pi(K)=\emptyset.
    \]
    We claim that in this case
    \begin{equation}\label{eq:outline-2}
         \|\chi_L R_M(\lambda\pm i0)\chi_R\|_{L^2\to L^2}\leq C\lambda^{-1}\sqrt{\log \lambda}.
    \end{equation}
    Decomposing
    \[
    \chi_R=\chi_0^2\chi_R+(1-\chi_0^2)\chi_R,
    \]
    we see that the second term is covered by the previous case \eqref{eq:outline-1}.
    Therefore, after modifying the cutoffs harmlessly, it suffices to
    treat the model case \(\chi_R=\chi_0^2\).

    We then compare the resolvent on \(M\) with the resolvent on the even
    manifold \(\tilde M\). The gluing identity
    \eqref{thm_1:keyID} gives
    \[
    \chi_L R_M(\lambda\mp i0)\chi_0^2
    =
    \chi_L R_M(\lambda\mp i0)\psi
    [-\Delta_M,\chi_0]
    \psi R_{\tilde M}(\lambda\mp i0)\chi_0,
    \]
    where \(\psi\) is a fixed cutoff supported away from \(\pi(K)\).

    The first factor satisfies
    \[
    \|\chi_LR_M(\lambda\mp i0)\psi\|_{L^2\to L^2}
    \lesssim \lambda^{-1}
    \]
    by the previous step, since both \(\chi_L\) and \(\psi\) are
    supported in the nontrapping region. Moreover,
    \[
    [-\Delta_M,\chi_0]\colon H^1_\lambda\to L^2
    \]
    has norm \(O(\lambda)\), where \(H^1_\lambda\) defined in
    \eqref{eq:sobolev}. Finally, the estimate for the even resolvent,
    together with Lemma~\ref{appendix:h1} and
    Lemma~\ref{lem:DV12}, controls
    \[
    \|\psi R_{\tilde M}(\lambda\mp i0)\chi_0\|_{ L^2\to H^1_\lambda} \leq C\lambda^{-1}\sqrt{\log \lambda}.
    \]
    Combining these three estimates gives the desired bound \eqref{eq:outline-2} in the
    mixed trapped--nontrapped case.

    \item
    It remains to consider the case in which both cutoffs meet the
    trapped region, and we want to show that
    \begin{equation}\label{eq:outline-3}
         \|\chi_L R_M(\lambda\pm i0)\chi_R\|_{L^2\to L^2}\leq C\lambda^{-1}\log \lambda.
    \end{equation}
    Using the same cutoff decomposition as above, we may
    reduce to the model case
    \[
    \chi_L=\chi_0^2,
    \quad
    \chi_R=\chi_0.
    \]
    In this case, the gluing identity takes the form
    \[
    \chi_0R_M\chi_0^2
    =
    \chi_0^2R_{\tilde M}\chi_0
    +
    \chi_0R_{ M}\psi
    [-\Delta_M,\chi_0]
    \psi R_{\tilde M}\chi_0.
    \]
    The first term on the right-hand side is controlled directly by the
    corresponding resolvent estimate on the even manifold
    \(\tilde M\), obtained from
    \cite{tao2021spectral} and Lemma~\ref{lem:XZ25-resolvent}.
    In the second term, the factor
    \(\psi R_M\chi_0\) is covered by \eqref{eq:outline-2},
    while the remaining factors are estimated exactly as in the
    preceding step. This completes the proof of the localized resolvent
    estimate \eqref{eq:outline-3}.
\end{itemize}

To prove Theorem~\ref{thm2}, we combine the resolvent estimate
\eqref{main:eq} with the abstract Strichartz and spectral-projection
machinery developed in \cite{Bounded_HSTZ}. The argument of \cite{Bounded_HSTZ} requires an auxiliary negatively
curved surface \((N,g_N)\), obtained by filling in the ends of \(M\)
with model hyperbolic discs, such that the Laplacian $\Delta_N$ has no eigenvalue and has no resonance at the bottom of the continuous spectrum. In \cite{Bounded_HSTZ}, this
property is established under the assumption that \(M\) is even. In
that setting, it follows readily from Vasy's method. Since Vasy's
method is not directly applicable to the non-even manifolds considered
here, we verify the required spectral property by a direct computation
in Section~\ref{sec:comparison}. Once this auxiliary construction has
been established, Theorem~\ref{thm2} follows by applying the abstract
results of \cite{Bounded_HSTZ} to the resolvent bound
\eqref{main:eq}.

This article is organized as follows. In Section~\ref{sec:prelim} we review the basic definitions and the results we use later. In Section~\ref{sec:CV} we verify that asymptotically hyperbolic ends satisfy the Cardoso--Vodev condition. In Section~\ref{sec:glue} we show that an asymptotically hyperbolic end can be glued into an \textit{even} asymptotically hyperbolic surface, and in Section~\ref{sec:thm1} we prove Theorem~\ref{Thm}. Section~\ref{sec:comparison} constructs the comparison surface $(N,g_N)$ and establishes the absence of eigenvalues and of resonances at the bottom of its continuous spectrum; this is the only new ingredient for Theorem~\ref{thm2}, and we will talk a little bit more in Section~\ref{sec:thm2}.

\subsection*{Acknowledgements}
The author would like to express his sincere gratitude to Zhongkai Tao for many illuminating discussions, from which the key idea in the proof of Theorem~\ref{Thm} emerged, and for his generous guidance throughout this project.
\section{Preliminaries}\label{sec:prelim}

\subsection{The Cardoso--Vodev result}
First, we recall some facts and assumptions in \cite{cardoso2002uniform}. Assume that $M=M_0\cup M_{\infty}$, where $M_0$ is compact  and $M_{\infty}$ is diffeomorphic to $[0,1)\times \partial M$. We introduce another parameter: set $r=-\log x$, then we assume on $(r,\theta)\in[0,+\infty)\times\partial M$, the metric $g$ satisfy \begin{equation}\label{an-para}
g\mid_{M_{\infty}} (r,\theta)=dr^2+\sigma(r) =dr^2+\sigma_{ij}(r,\theta) d\theta_id\theta_j,\quad \sigma(r) = e^{2r}h(e^{-r}).
\end{equation}
Let $\sigma^{ij}(r,\theta)$ be the inverse of $\sigma_{ij}(r,\theta)$, and write
$\eta(r,\theta,\xi)=\sum_{i,j}\sigma^{ij}(r,\theta)\xi_i\xi_j$. Let $p=(\det \sigma_{ij})^{-1/2}$ and 
\[
q=\frac{(\partial_r p)^2}{(2p)^2}+\frac{1}{(2p)^2}\sum_{i,j} \sigma^{ij}(\partial_{\theta_i}p)(\partial_{\theta_j}p)+\frac{1}{2}p\Delta_{M_{\infty}}(p^{-1}).
\]
We write $R_M(\lambda)=(-\Delta_g-\frac{(n-1)^2}{4}-(\lambda\pm i0)^2)^{-1}$ for the limiting resolvent (the choice of $\pm$ is fixed throughout). 
We recall the following theorem of Cardoso--Vodev \cite{cardoso2002uniform}:
\begin{prop}\label{prop:cv}
  Suppose $(M,g)$ is a complete manifold as above. Suppose there exist constants $C,\delta_0,r_0>0$ such that for every $r\geq r_0$, $\theta\in \partial M$, $\xi\in T^*_\theta \partial M$:
\begin{align}
|q(r,\theta)|&\leq C, \label{CV-q-bound}\\
\frac{\partial q}{\partial r}(r,\theta)&\leq C\,r^{-1-\delta_0}, \label{CV-q-deriv}\\
-\frac{\partial \eta}{\partial r}(r,\theta,\xi)&\geq \frac{C}{r}\,\eta(r,\theta,\xi). \label{CV-h-deriv}
\end{align}
Then there exists $a>r_0$ such that for any compactly supported function $\chi\in C_c^{\infty}(M)$ with $\supp \chi \subset M_a:=\{r>a\}$, we have 
\begin{equation}
    \|\chi R_M(\lambda\pm i0)\chi\|_{L^2(M)\to L^2(M)}\leq C\lambda^{-1}.
\end{equation}
\end{prop}
\begin{rem}\label{coordinate}
    In the following context, we use $(x,y)$ to represent the coordinate $[0,1)\times \partial M$ and $(r,\theta)$ to represent the coordinate $[r_0,+\infty)\times \partial M$.
\end{rem}

\subsection{Asymptotically hyperbolic manifold}\label{sec:CV}

\begin{defi}
\label{def:AH}
Let $\overline{M}$ be a smooth $n$-dimensional manifold with non-empty boundary $\partial{M}$, and $M=\operatorname{Int}(\overline{M})$. A complete Riemannian metric $g$ on $M$ is asymptotically hyperbolic (AH) if there exists a boundary defining function $x\in C^\infty(\overline{M})$ (i.e., $x\geqslant 0$, $x=0$ on $\partial{M}$, and $dx\neq 0$ on $\partial{M}$) such that, in a collar neighborhood of $\partial{M}$,
\begin{equation}\label{eq:ah}
g=\frac{dx^2+h(x,x',dx')}{x^2},
\end{equation}
where $h(x,x',dx')$ is a smooth one-parameter family of metrics on $\partial{M}$. We say $g$ is even asymptotically hyperbolic if $h(x,x',dx')$ admits a Taylor expansion at $x=0$ in even powers:
\[
h(x,x',dx')\sim h_0+x^2 h_2+x^4 h_4+\cdots, \quad x\to 0^+.
\]
Equivalently, $h$ extends to an even smooth function of $x$ across $x=0$.
\end{defi}

The convexity conditions in Proposition~\ref{prop:cv} hold for asymptotically hyperbolic manifolds.
\begin{prop}
\label{ass:CV}
Let $(M,g)$ be an asymptotically hyperbolic manifold of dimension $n\geq 2$. Then \eqref{CV-q-bound}, \eqref{CV-q-deriv} and \eqref{CV-h-deriv} hold.
\end{prop}
\begin{proof}
From \eqref{an-para}, we know for $h$ in \eqref{eq:ah}
\begin{equation}\label{sigma^{ij}}
\sigma_{ij}(r,\theta)=e^{2r}h_{ij}(e^{-r},\theta), \quad
\sigma^{ij}(r,\theta)=e^{-2r}h^{ij}(e^{-r},\theta),
\end{equation}
where $h_{ij}(x,\theta)$ is smooth up to $x=0$ and $h^{ij}$ is its inverse. By Taylor's theorem,
\begin{equation}
    \partial_\theta^\alpha(h^{ij}(x,\theta)-h^{ij}(0,\theta))=
    x\int_0^1\partial_x\partial_\theta^\alpha
h^{ij}(tx,\theta)\,dt.
\end{equation}
Hence, for $r$ large, we have
\begin{equation}\label{eq:asym}
    \partial_r^k\partial_\theta^\alpha\bigl(h^{ij}(e^{-r},\theta)-h^{ij}(0,\theta)\bigr)\le C_{\alpha,k}\, e^{-r},\quad r\to\infty,
\end{equation}
for every $k\in\bbZ$ and $\alpha$ multiple index. Similarly, by the expression in \eqref{sigma^{ij}},
we have
\[
|\partial_r^k\partial_\theta^\alpha\sigma^{ij}(r,\theta)|
\leq C_{k,\alpha}e^{-2r}.
\]

By \eqref{sigma^{ij}}, we have
\[
\eta(r,\theta,\xi)=\sum_{i,j}\sigma^{ij}(r,\theta)\xi_i\xi_j=e^{-2r}\sum_{i,j}h^{ij}(e^{-r},\theta)\xi_i\xi_j.
\]
We know 
\[
\partial_r \eta(r,\theta,\xi)=
-2\eta+e^{-2r} \sum_{i,j}(\partial_rh^{ij})\xi_i\xi_j.
\]
hence
\[
|\partial_r h^{ij}(e^{-r},\theta)|=|e^{-r}(\partial_x h^{ij})(e^{-r},\theta)|\le Ce^{-r},
\]
which gives
\[
|e^{-2r}\sum_{i,j}(\partial_rH^{ij})\xi_i\xi_j|
\le Ce^{-3r}|\xi|^2\leq
Ce^{-r}\eta.
\]
so if we choose $r$ large enough, we have
\[
-\partial_r\eta \geq (2-Ce^{-r})\eta\ge \eta\ge \frac{C}{r}\eta.
\]
which is the \eqref{CV-h-deriv}.

For $p(r,\theta)$, we have
\[
p(r,\theta)=(\det \sigma_{ij})^{1/2}=e^{r\cdot (n-1)}\omega(r,\theta),\quad \omega(r,\theta)=(\det h_{ij}(e^{-r},\theta))^{1/2}.
\]
Since $\omega(r,\theta)$ is positive definite, so there exist $c,C>0$ such that
\[
c<\omega(r,\theta)<C.
\]
By the same argument as in \eqref{eq:asym},
\begin{equation}\label{eq:omega-r}
|
\partial_r^k\partial_\theta^\alpha(\omega(r,\theta)-\omega_0(\theta))
|
\leq C_{k,\alpha}e^{-r},\quad \omega_0(\theta)=\lim_{r\to +\infty}\omega(r,\theta).
\end{equation}
So we have
\begin{equation}
\label{eq:log-derivs}
\frac{\partial_r p}{p}=n-1+\frac{\partial_r\omega}{\omega},\quad
\frac{\partial_{\theta_i} p}{p}=\frac{\partial_{\theta_i}\omega}{\omega}=\frac{\partial_{\theta_i}(\omega_0+\omega-\omega_0)}{\omega},
\end{equation}
Hence we have
\begin{equation}
    \left|\frac{\partial_{\theta_i} p}{p}\right|\le C(1+e^{-r}),\quad \left|\frac{\partial_r p}{p}-(n-1)\right|\le Ce^{-r}
\end{equation}

Now, we compute the three components of the expression of $q$. For the first term
\begin{equation}\label{first_term}
    \left|\frac{(\partial_r p)^2}{(2p)^2}\right|=\left|\frac{1}{4}\Bigl(\frac{\partial_r p}{p}\Bigr)^2\right|\le \frac{(n-1)^2}{4}+C e^{-r}.
\end{equation}
using $\sigma^{ij}=e^{-2r}h^{ij}$, we have
\[
\left|\frac{1}{(2p)^2}\sum_{i,j}\sigma^{ij}\,(\partial_{\theta_i}p)(\partial_{\theta_j}p)\right|
=C\left|\sum_{i,j}\sigma^{ij}\bigl(\partial_{\theta_i}\log p\bigr)\bigl(\partial_{\theta_j}\log p\bigr)\right|\le Ce^{-2r}.
\]
For the last term, in coordinates, we have 
\[
\Delta_g(u)=\frac{1}{p}\partial_r(p\partial_ru)+\frac{1}{p}\sum_{i,j}\partial_{\theta_i}(p\sigma^{ij}\partial_{\theta_j}u).
\]
Then take $u=p^{-1}$, we have
\[
\frac{1}{p}\partial_r(p\partial_r(p^{-1}))=-\frac{1}{p}\partial_r\left(\frac{\partial_r p}{p}\right)=-\frac{\partial^2_rp}{p^2}+\frac{(\partial_r p)^2}{p^3},
\]
and
\[
\frac{1}{p}\sum_{i,j}\partial_{\theta_i}\left(p\sigma^{ij}\partial_{\theta_j}(p^{-1})\right)=-\frac{1}{p}\sum_{i,j}\partial_{\theta_i}\left(\sigma^{ij}\frac{\partial_{\theta_j}p}{p}\right). 
\]
Multiply with $\frac{1}{2}p$, we have 
\begin{equation}\label{eq:laplacian}
    \frac{1}{2}p\Delta_g(p^{-1})=-\frac{1}{2}\partial_r\left(\frac{\partial_rp}{p}\right)-\frac{1}{2}\sum_{i,j}\partial_{\theta_i}(\sigma^{ij}\partial_{\theta_j}\log p).
\end{equation}
The first contribution is $|\partial^2_r(\log \omega)|\le C e^{-r}$. The second term by \eqref{sigma^{ij}} and \eqref{eq:log-derivs}, hence is controlled by $e^{-r}$.

The decay of $q$, whose computation is similar:
\begin{equation}
        \frac{\partial q}{\partial r}=\partial_r\left(\frac{\partial_r p}{2p}\right)^2+\partial_r\left(C\sum_{i,j}\sigma^{ij}\bigl(\partial_{\theta_i}\log p\bigr)\bigl(\partial_{\theta_j}\log p\bigr)\right)
        +\partial_r \left( \frac{1}{2}p\Delta_g(p^{-1}) \right)
\end{equation}
by \eqref{first_term} and \eqref{eq:log-derivs}, we know after taking derivative with respect to $r$, we have $\partial_r(\partial_rp/2p)$ and $\partial_r(\partial_rp/2p)^2$ are controlled by $e^{-r}$. By the second equality of \eqref{eq:log-derivs} and \eqref{sigma^{ij}}, we know 
\[
\left|\sum_{i,j}\sigma^{ij}\bigl(\partial_{\theta_i}\log p\bigr)\bigl(\partial_{\theta_j}\log p\bigr)\right|\le Ce^{-r},
\]
so by taking the derivative with respect to $r$. The last term, by \eqref{eq:laplacian}, is bounded by $e^{-r}$. So in all, the decay of $q$ is controlled by $e^{-r}$ hence $r^{-1-\delta_0}$ for any $\delta_0>0$.
\end{proof}

\subsection{Gluing to an even asymptotically hyperbolic surface}\label{sec:glue}

In this section we show that any asymptotically hyperbolic surface can be embedded into an even asymptotically hyperbolic one.
\begin{prop}\label{prop:glue}
Let $(M,g)$ be an asymptotically hyperbolic surface, $M=M_0\cup M_\infty$, with $M_0$ compact and $M_\infty$ an asymptotically hyperbolic end, and the trapped set $K\subset M$(see Proposition~\ref{prop:trapped_compact}). Then there exists $r_1>r_0>0$, where $r_0$ is the $r_0$ in Proposition~\ref{ass:CV}, and there exists an \textit{even} asymptotically hyperbolic surface $(\tilde{M},\tilde{g})$ such that:
\begin{enumerate}
\item[\textup{(a)}] $\tilde{M}=M_0\cup\tilde{M}_\infty$, where $\tilde{M}_\infty$ has the same underlying smooth manifold as $M_\infty$;
\item[\textup{(b)}] $\tilde{g}=g$ on a neighborhood of $M_0\cup\{(r,\theta)\in M_{\infty}: r<r_1\}$ inside $\tilde{M}$, and the trapped set $\tilde K$ of $\tilde{M}$ is equal to $K$.
\item[\textup{(c)}] $\tilde{g}$ is even AH near $\partial{\tilde{M}}$.
\end{enumerate}
If in addition $\sec_g\leq -k_0^2<0$ on $M$, then $\tilde{g}$ can be chosen to satisfy $\sec_{\tilde{g}}\leq -k_1^2<0$ globally on $\tilde{M}$ for some $k_1\in(0,k_0]$.
\end{prop}
\begin{proof}
We will use the notations as in Remark~\ref{coordinate}.
On $M_\infty$ choose a boundary defining function $x$ and a collar diffeomorphism $\Phi:[0,\varepsilon_0)_x\times \partial{M}\to U\subset{M}$ such that $\Phi^*g$ has the form $(dx^2+h(x))/x^2$ for some smooth family of metrics $h(x)$ on $\partial{M}$. We assume $\varepsilon_0<e^{-r_0}$ is small.

Let $h^{\rm ev}(x)$ be any smooth one-parameter family of metrics on $\partial M$ that is \emph{even} in $x$, i.e., a smooth function of $x^2$, and that satisfies $h^{\rm ev}(0)=h(0)$. We will choose the constant family $h^{\rm ev}(x)\equiv h(0)$.

Pick a fixed constant(which will be chosen later) $\varepsilon_1\in(0,\varepsilon_0)$ small enough that $M_0$ is properly contained in $\Phi([\varepsilon_1,\varepsilon_0)\times \partial M)\cup M_0$. Choose $\chi\in C_c^\infty(M)$ with
\[
\chi(x)=1\text{ for }x\ge \frac{2}{3}\varepsilon_1,\quad\chi(x)=0\text{ for }x\le \frac{1}{3}\varepsilon_1.
\]
Define
\begin{equation}
\label{eq:def-htilde}
\tilde{h}(x):=\chi(x)\,h(x)+(1-\chi(x))\,h^{\rm ev}(x)\quad\text{for }x\in[0,\varepsilon_0).
\end{equation}
This is a smooth family of metrics on $\partial M$, and:
\begin{itemize}
\item For $x\geqslant \frac{2}{3}\varepsilon_1$: $\tilde{h}(x)=h(x)$.
\item For $x\leqslant \frac{1}{3}\varepsilon_1$: $\tilde{h}(x)=h^{\rm ev}(x)$, even with respect to $x$.
\end{itemize}

Define $\tilde{g}$ on $M$ by:
\[
\tilde{g}:=\begin{cases}g & \text{on }M\setminus\Phi([0,\tfrac{2}{3}\varepsilon_1)\times S),\\[2pt]
\Phi^{-1\,*}\!\Bigl(\dfrac{dx^2+\tilde{h}(x)}{x^2}\Bigr) & \text{on }\Phi([0,\varepsilon_0)\times S).
\end{cases}
\]
The two definitions agree on the overlap $\Phi([\tfrac{2}{3}\varepsilon_1,\varepsilon_0)\times \partial M)$ because $\tilde{h}=h$ there. Thus $\tilde{g}$ is a well-defined smooth complete metric on $M$ that is even asymptotically hyperbolic near infinity. By construction $\tilde{g}=g$ on a neighborhood of $M_0$, and $\tilde{h}$ is even in $x$ near $x=0$, so $\tilde{g}$ is even AH. Set $\tilde{M}=M$ with this new metric. Once we show that two trapped sets are equal, properties (a)--(c) are immediate. 

To show this, we will show that for $\varepsilon_1$ small, we have
\[
-\partial_r\eta(x,\theta)>\eta(x,\theta),\quad \forall\, (x,\theta)\in \{(x,\theta)\colon x<\varepsilon_1\}
\]

Since we are in dimension $2$, so $\eta(r,\theta)=e^{-2r}\tilde{h}(e^{-r},\theta)$. Where
\[
\tilde{h}(x,y)dy^2=\tilde h(x).
\]
By direct computation, we have
\begin{equation}
\begin{split}
     \partial_r\eta&=-2e^{-2r}\tilde{h}(e^{-r},\theta)+e^{-2r}\partial_r(\tilde{h}(e^{-r},\theta))= -2\eta+e^{-2r}\partial_r(\tilde{h}(e^{-r},\theta))\\
     &=-2\eta+\partial_r[\chi(e^{-r})\cdot(h(e^{-r},\theta)-h(0,\theta))]\\
     &=-2\eta+e^{-3r}(h(e^{-r},\theta)-h(0,\theta))\chi'(e^{-r})+e^{-3r}\chi\cdot(\partial_xh(e^{-r},\theta))
\end{split}
\end{equation}

By $\chi'\sim \varepsilon_1$ the same argument in Proposition~\ref{ass:CV}, we have
\begin{equation}
    -\partial_r\eta\ge (2-Ce^{-r})\eta.
\end{equation}
Where $C$ is independent from $\varepsilon_1$. We can choose $\varepsilon_1$ small enough, and since $C$ is independent from $\varepsilon_1$, for $r_1$ large enough such that $e^{-r_1}>\frac{2}{3}\varepsilon_1$, we have $2-Ce^{-r_1}>1$. By Lemma~\ref{escape} and Proposition~\ref{prop:trapped_compact}, we have $\tilde K\subset \{r<r_1\}=\{x>-\log r_1\}$, hence $\tilde K\subset \{x>\frac{2}{3}\varepsilon_1\}$. But in $\{x>\frac{2}{3}\varepsilon_1\}$, two metrics coincide, hence $\tilde K=K$.

In coordinates, suppose $g$ has the expression 
\[
g=\frac{dx^2+a(x,y)^2dy^2}{x^2}\quad\Longrightarrow \quad \tilde{g}=\frac{dx^2+\tilde{a}(x,y)^2dy^2}{x^2},
\]
with $\tilde{a}^2=\chi a^2+(1-\chi)a^2(0,y)$. In this expression, we have the Gauss curvature
\[
K_g=-1+\frac{xa_x}{a}-\frac{x^2a_{xx}}{a}.
\] 
We need to compute the curvature in the transition region, which is close to the boundary, so we can assume $a>\varepsilon_0>0$ in this region. Set $\rho(x,y)=a^2(x,y)-a^2(0,y)$, which vanishes on the boundary, hence $\rho(x,y)=x_1\rho_1(x,y)$ for $\rho_1$ smooth function. Then
\[
\tilde{a}^2(x,y)=a^2(0,y)+R(x,y),\quad R(x,y)=\chi(x)x\rho_1(x,y).
\]
Since the curvature form only have the derivative w.r.t. $x$, so we take differential:
\begin{align*}
R_x   &= \chi\, \rho_1 + x\, \chi'\, \rho_1 + x\, \chi\, \rho_{1,x}, \\
R_{xx} &= 2\chi'\rho_1 + 2\chi\,\rho_{1,x}
        + x\,\chi''\rho_1 + 2x\,\chi'\rho_{1,x} + x\,\chi\,\rho_{1,xx}.
\end{align*}
The cut-off $\chi$, satisfies $|\chi'| \leqslant C\varepsilon_1^{-1}$ and $|\chi''| \leqslant C\varepsilon_1^{-2}$ on the transition interval $[1/3\varepsilon_1,2/3\varepsilon_1]$. At the same time, $x\sim \varepsilon_1$. So, 
\[
|R| = O(\varepsilon_1),\quad
|R_x| = O(1),\quad
|R_{xx}| = O(\varepsilon_1^{-1}).
\]
and
\[
|xR_x|\sim \varepsilon_1,\quad |x^2R_{xx}|\sim \varepsilon_1.
\]
So we can choose $\varepsilon_1$ small enough so that the sectional curvature stays below a negative constant on the transition region.

\end{proof}

\subsection{Propagation estimates}
In this section, we recall some basics of the elliptic and propagation estimates (see \cite[Appendix E]{dyatlov2019mathematical} for more details).

Write $h=\lambda^{-1}$, set 
\[
P_h:=-h^2\Delta_g-\tfrac{h^2}{4}-1,\quad E:=|\xi|_g^2-1,\quad \Sigma:=E^{-1}(0)=S^*M,
\]
Define the flow $\varphi_t=\operatorname{exp}(tH_p)$ to be the geodesic flow.
The backward/forward trapped set is defined as
\begin{equation}
\Gamma_\pm:=\{\rho\in\Sigma: \{\varphi_{\mp t}(\rho):t\ge0\}\ \text{bounded}\}.
\end{equation}
The trapped set $K$ is defined by
\begin{equation}\label{eq:trapped-set}
    K:=\Gamma_+\cap \Gamma_-.
\end{equation}
\begin{prop}\label{prop:trapped_compact}
    Under our assumptions in Theorem~\ref{Thm}, the forward and backward trapped sets are both closed, and the trapped set $K$ is a compact set.
\end{prop}
\begin{proof}
    We will only prove for $\Gamma_-$, \,$\Gamma_+$ is similar.
    We use the parameter in \eqref{an-para}, via direct computation, we have
    \begin{equation}
        H_{\|\xi\|^2-1}(r,\theta,\nu,\xi)=2\nu\partial_r+2\sigma^{ij}\xi_j\partial_{\theta_i}-\frac{\partial \eta}{\partial r}\partial_\nu -\frac{\partial \eta}{\partial \theta_i}\partial_{\xi_j}.
    \end{equation}
    Hence, if a curve $(r,\theta,\nu,\xi)$ is an integral curve of this Hamilton vector field, we have
    \begin{equation}\label{eq:trapped_compact_esti}
        \dot{r}=2\nu,\quad \dot{\nu}=-\frac{\partial \eta}{\partial r}\ge \frac{C}{r}\eta\ge 0.
    \end{equation}

    Define $O_+=\{(r,\theta,\nu,\xi)\in \Sigma\colon r>R>r_0,\nu> 0\}$, where $R$ is a fixed number, by \eqref{eq:trapped_compact_esti}, we know $O_+\cap \Gamma_-=\emptyset$. If an integration curve $\gamma(t)=(r,\theta,\nu,\xi)(t)$ is not contained in $\Gamma_-$, then there exists a first time
$t_0>0$ such that
\[
\gamma(t_0)\in\{r=r_0\}.
\]
  Furthermore, we have $\nu(t_0)\ge 0$. We claim the following lemma:
  \begin{lem}\label{escape}
      For such a $\gamma$, which is a integration curve and does not contained in $\Gamma_-$, then there exist a $t_1>t_0$, such that 
  \begin{equation}
      \gamma\mid_{t>t_1}\subset O_+.
  \end{equation}
  \end{lem}
  \begin{proof}
    There exist a first time $t_1$ of $\gamma$ intersect $\{r=R\}$. It suffices to show that such a first time is the desired $t_1$. For $t_0$ defined above, we have $\dot{r}(t_0)=\nu\ge0$. By the equation \eqref{eq:trapped_compact_esti}, we have the following three cases:
    \begin{itemize}
        \item If $\nu(t_0)\neq 1,0$, then by \eqref{eq:trapped_compact_esti}, we know $\dot{r}$ and $\dot{\nu}$ will always be larger than $0$ for $t>t_0$.
        \item If $\nu=1$, then $\nu=1$ for $t>t_0$.
        \item If $\nu=0$, then $\dot{\nu}=\frac{C}{r}\eta>0$, hence $\nu>0$ for $t>t_0$. \qedhere
    \end{itemize}  
  \end{proof}
By this claim, we have that 
\begin{equation}
    \Gamma_-=\Sigma\setminus \bigcup_{T\in \bbR}\varphi_T(O_+)\quad \Longrightarrow \quad \Gamma_-\text{ is closed.}
\end{equation}
Since $\varphi_T(O_+)$ is open.

Finally, the compactness of $K$. Since $\Gamma_-$ and $\Gamma_+$ are closed, we have that $K=\Gamma_-\cap \Gamma_+$ is also closed. We can see that $K\subset \{(r,\theta,\nu,\xi)\in \Sigma\colon r<R\}$: if a point in $(r,\theta,\nu,\xi)\in \{(r,\theta,\nu,\xi)\in \Sigma\colon r<R\}$. Then by Lemma~\ref{escape}, we know the integral curve with initial value $(x,\nu)$ will go to infinity forward or backward. But $\{(r,\theta,\nu,\xi)\in \Sigma\colon r<R\}$ is compact, this shows that $K$ is compact.
\end{proof}

For $u\in L^2(M)$ and $s\in \mathbb{R}$, we use the semiclassical Sobolev norm
\begin{equation*}
    \|u\|_{H_h^s}:= \| (1-h^2\Delta)^{s/2} u\|_{L^2}.
\end{equation*}
We also introduce the notation
\begin{equation}\label{eq:sobolev}
    \|u\|_{H^1_\lambda}^2=\|u\|_{L^2}^2+\lambda^{-2}\|\nabla u\|_{L^2}^2.
\end{equation}
We will use the following elliptic estimate and propagation estimate (see \cite[Appendix E]{dyatlov2019mathematical} for more details):
\begin{lem}[Elliptic estimate]\label{lem:elliptic}
     If $A\in\Psi^{0}_h$, $B\in\Psi_h^0$ is elliptic on $\WF_h(A)$, and $\operatorname{WF}_h(A)\cap\Sigma=\emptyset$,
then for any $s\in \mathbb{R}$, we have
\begin{equation}
    \|Au\|_{H^s_h}\le C\|B P_h u\|_{H^{s-2}_h}+O(h^\infty)\|\chi u\|_{H^{-N}_h}. 
\end{equation}
Where $\chi\in C_c^\infty(M)$ depends on $A,B,P_h$.
\end{lem}

\begin{lem}[Propagation estimate]\label{lem:propagation}
    Let $A,B,B_1\in\Psi^{\mathrm{comp}}_h$ with
$\operatorname{WF}_h(A)\subset\operatorname{Ell}(B_1)$, and suppose that for each $\rho\in\operatorname{WF}_h(A)$ there is $T(\rho)\ge0$
with $\varphi_{-T(\rho)}(\rho)\in\operatorname{Ell}(B)$ and $\varphi_{-t}(\rho)\in\operatorname{Ell}(B_1)$ for
$0\le t\le T(\rho)$. Then for any $s\in \mathbb{R}$,
\begin{equation}\label{eq:propagation_estimate}
\|Au\|_{H_h^s}\le C\|Bu\|_{H_h^s}+Ch^{-1}\|B_1 P_h u\|_{H^{s-1}_h}+O(h^\infty)\|\tilde \chi u\|_{H^{-N}_h} .
\end{equation}
Where $\tilde{\chi}\in C_c^\infty(M)$ is supported in a neighbourhood of the support of the Schwartz kernel of $A,B,B_1$.
\end{lem}

\section{Proof of Theorem~\ref{Thm}}\label{sec:thm1}

We use without proof the following two estimates as black boxes.
\begin{lem}[Even AH log-loss resolvent \cite{tao2021spectral}]
\label{lem:XZ25-resolvent}
For every $\chi\in C_c^\infty(\tilde{M})$, where $\tilde{M}$ is an even asymptotically hyperbolic surface with negative curvature, there exists $C>0$ such that for $\lambda\geqslant C$,
\begin{equation}
\label{eq:XZ-bound}
\|\chi R_{\tilde{M}}(\lambda\pm i0)\chi\|_{L^2(\tilde{M})\to L^2(\tilde{M})}\leq C\lambda^{-1}\log\lambda.
\end{equation}
\end{lem}

\begin{lem}[Half-log improvement away from trapping \cite{DatchevVasy2012}]
\label{lem:DV12}
Suppose $\tilde{M}$ is an even asymptotically hyperbolic manifold satisfying \eqref{eq:XZ-bound}.
Let $\chi,\chi_1\in C_c^\infty(\tilde M)$ with $\chi$ supported in any neighborhood of the trapped set and $\chi_1$ supported in the non-trapping region. Then for $\lambda\geqslant C$,
\begin{equation}
\label{eq:DV-bound}
\|\chi_1 R_{\tilde M}(\lambda\pm i0)\chi\|_{L^2(\tilde{M})\to L^2(\tilde{M})}\leqslant C\lambda^{-1}\sqrt{\log\lambda}.
\end{equation}
\end{lem}

In addition we use the Cardoso--Vodev estimate from Proposition~\ref{prop:cv}: for any $\chi_\infty\in C_c^\infty(M_{a})$ supported sufficiently far in the non-trapping part,
\[
\|\chi_\infty R_M(\lambda\pm i0)\chi_\infty\|_{L^2\to L^2}\leqslant C\lambda^{-1}, \quad \lambda\geqslant C.
\]

\begin{proof}[Proof of Theorem~\ref{Thm}]

We set
\begin{equation*}
    u=(-\Delta_g-(n-1)^2/4-(\lambda\pm i0)^2)^{-1}\chi f.
\end{equation*}

\emph{Key estimate} Let $U_C \subset \{x\in M \text{ such that }g(x)=\tilde{g}(x) \}$ be a neighborhood of $\pi(K)$. And $M_a$ as in Proposition~\ref{prop:cv}. Choose any neighborhood $U_I$ of $M\setminus(M_a\cup M_C)$. Then the set $\{U_C, U_I,M_a\}$ forms an open cover of $M$, choose a partition of unity subordinate to this open cover, called $\{\chi_0,\chi_1,\chi_2\}$, hence we have
\[
1=\chi_0+\chi_1+\chi_2.
\]
We can choose another partition of unity $\eta_0,\eta_1,\eta_2$ such that 
\begin{equation}\label{eq:p.o.u}
    1=\sum_i\eta_i,\quad\chi_0=\eta_0, \chi_1\le \eta_1,\quad \text{and }\operatorname{supp}\eta_2\subset \{\chi_2=1\}. 
\end{equation}
We may also without loss of generality assume $\pi(K)\subset \operatorname{Int}\{\chi_0=1\}$.

\begin{prop}\label{prop:propagation}
    Let $\chi_L\in C_c^\infty(M)$ be any cut-off function such that $\supp \chi_:$ is disjoint from $\pi(K)$, then there exist $C>0$, $\chi\in C_c^\infty(M)$, and  $\chi_\infty\in C_c^\infty(M_a)$, which depends on $\chi_L$, such that $\operatorname{supp} \chi_\infty \subset \{\chi_2=1\}$, we have
    \begin{equation}\label{eq:propagation}
\|\chi_{L} u\|_{L^2} \lesssim \|\chi_\infty u\|_{L^2}  +\lambda^{-1}\|\chi f\|_{L^2},\quad \lambda\ge C.
\end{equation}
\end{prop}
This proposition will be proved later.

Since $\chi_L u=\chi_LR_M(\lambda\pm i0)\chi f$, \eqref{eq:propagation} can be translated into the following:
\begin{equation}\label{thm1:key_estimate}
    \|\chi_LR_M(\lambda\pm i0)\chi\|_{L^2\to L^2}\lesssim \|\chi_\infty R_M (\lambda\pm i0)\chi\|_{L^2\to L^2}+\lambda^{-1}.
\end{equation}
    By taking the adjoint, we also have
\begin{equation*}
     \|\chi R_M (\lambda\mp i0)\chi_L\|_{L^2\to L^2}\lesssim \| \chi R_M(\lambda\mp i0)\chi_\infty\|_{L^2\to L^2}+\lambda^{-1}.
\end{equation*}    

We split the proof of Theorem~\ref{Thm} into three Propositions:

\begin{prop}[Nontrapping Case]\label{prop:Nontrappingcase}
    Suppose $\chi_L,\chi$ are all disjoint from $\pi(K)$. Then we have
\[
\|\chi_LR_M(\lambda\pm i0)\chi\|_{L^2\to L^2}\lesssim \lambda^{-1}.
\]
\end{prop}
\begin{proof}
Suppose $\chi_L,\chi$ are all disjoint from $K$. Then we use \eqref{thm1:key_estimate},
we have
\[
\|\chi_LR_M(\lambda\pm i0)\chi\|_{L^2\to L^2}\lesssim \|\chi_\infty R_M(\lambda\pm i0)\chi\|_{L^2\to L^2}+\lambda^{-1}.
\]
By taking adjoint and use \eqref{thm1:key_estimate} again we have
\begin{equation}
    \|\chi_\infty R_M(\lambda\pm i0)\chi\|_{L^2\to L^2}=\|\chi R_M(\lambda\mp i0)\chi_\infty\|_{L^2\to L^2}\lesssim \|\chi_\infty R_M(\lambda\mp i0)\chi_{\infty,1}\|_{L^2\to L^2}+\lambda^{-1}.
\end{equation}
Then we use Proposition~\ref{prop:cv}, we have
\begin{equation}
    \|\chi_LR_M(\lambda\pm i0)\chi\|_{L^2\to L^2}\lesssim \lambda^{-1}.
\end{equation}
\end{proof}

\begin{prop}[Half-trapped case]\label{prop:Half-trappedcase}
     Suppose $\chi_L,\chi_R$ are two cutoff functions. We assume $\operatorname{supp}\chi_L\cap K\neq\emptyset$ and $\operatorname{supp}\chi_R\cap K=\emptyset$, then we have: 
\[
\|\chi_LR_M(\lambda\pm i0)\chi_R\|_{L^2\to L^2}\le \lambda^{-1}\sqrt{\log \lambda}.
\]
\end{prop}
\begin{proof}
We first consider $ \chi_0^2R_M \tilde{\chi_2}$, where $\tilde{\chi_2}\in C_c^\infty(M_a)$. Before we go into the proof, we state one lemma:
\begin{lem}\label{lem:comparision}
    Suppose $\kappa$ is a compact-supported cutoff function, and $\chi_0$ is the cutoff function defined before, then we have the following identity of resolvents
    \begin{equation}
        \chi_0R_M(\lambda\pm i0)\kappa= R_{\tilde{M}}\chi_0\kappa + \left(R_{\tilde M}[-\Delta_M,\chi_0]R_{M}\kappa\right)(\lambda\pm i0).
    \end{equation}
\end{lem}
\begin{proof}
    Take $u=R_{ M}\kappa f$, then 
    \begin{equation}
\begin{split}
    (-\Delta_{\tilde M}-(n-1)^2/4-(\lambda\pm i0)^2)\chi_0u &=(-\Delta_{ M}-(n-1)^2/4-(\lambda\pm i0)^2)\chi_0 u\\
    &=\chi_0 (-\Delta_{M}-(n-1)^2/4-(\lambda\pm i0)^2)u+[-\Delta_M,\chi_0]u\\
    &=\chi_0\kappa f +[-\Delta_M,\chi_0]R_{M}\kappa f.
\end{split}
\end{equation}
By taking $R_{\tilde{M}}$, hence,
\begin{equation}
    \chi_0R_M(\lambda\pm i0)\kappa= R_{\tilde{M}}\chi_0\kappa+\left(R_{\tilde M}[-\Delta_M,\chi_0]R_{M}\kappa\right)(\lambda\pm i0).
\end{equation}
\end{proof}

By taking $\kappa=\tilde{\chi_2}$, we have
\begin{equation}\label{thm1:compare}
    \chi_0R_M(\lambda\pm i0)\tilde{\chi_2}= \left(R_{\tilde M}[-\Delta_M,\chi_0]R_{M}\tilde{\chi_2}\right)(\lambda\pm i0).
\end{equation}
Note that $[-\Delta_M,\chi_0]$ is support on $\{\chi_0'\neq 0\}$. Since we have assume that $K\subset \{\chi_0=1\}$, so there exist a $\psi$ which is $1$ on $\{\chi_0'\neq 0\}$, and $\operatorname{supp}\psi$ is disjoint from $K$. Then we have
\begin{equation}\label{thm_1:keyID}
    \chi_0^2R_M(\lambda \pm i0)\tilde{\chi_2}=(\chi_0R_{\tilde M}\psi [-\Delta_M,\chi_0]\psi R_{M}\tilde{\chi_2})(\lambda \pm i0).
\end{equation}
By taking adjoint, we have 
\begin{equation}\label{thm_1:keyID_adj}
    \tilde{\chi_2} R_M(\lambda\mp i0)\chi_0^2=-(\tilde{\chi_2} R_{{M}}\psi [-\Delta_M,\chi_0]\psi R_{\tilde M}\chi_0)(\lambda\mp i0).
\end{equation}
Our goal is to analyze the right-hand sides of \eqref{thm_1:keyID_adj}.

We choose a $\psi_1$ with compact support such that $\psi_1=1$ on $\operatorname{supp}\psi $ which is disjoint from $K$. Then
\[
\tilde{\chi_2} R_{{M}}\psi [-\Delta_M,\chi_0]\psi R_{\tilde M}\chi_0(\lambda\mp i0)=\tilde{\chi_2} R_{{M}}\psi_1 [-\Delta_M,\chi_0]\psi_1 R_{\tilde M}\chi_0(\lambda\mp i0)
\]
\begin{itemize}
    \item \emph{Term $\psi_1 R_{\tilde M}\chi_0$}
By Lemma~\ref{lem:DV12}, we have
\begin{equation*}
    \|\psi_1 R_{\tilde M}(\lambda\mp i0)\chi_0\|_{L^2\to L^2}\leqslant C\lambda^{-1}\sqrt{\log \lambda}.
\end{equation*}
By same computation in Lemma~\ref{appendix:h1} below applied to $\tilde{M}$, we have
\begin{equation}
    \|\psi R_{\tilde M}(\lambda\mp i0)\chi_0\|_{L^2\to H^{1}_\lambda}\leqslant C\lambda^{-1}\sqrt{\log \lambda}.
\end{equation}

\item\emph{Term $[\Delta,\chi_0]$}
The operator $[-\Delta, \chi_0]$ is a first order differential operator. So 
\begin{equation}\label{eq:esLaplacian}
    \begin{split}
     \|[\Delta, \chi_0]u\|^2_{L^2}&< C\|u\|_{H^1}^2=C(\|u\|_{L^2}^2+\|\nabla u\|_{L^2}^2)\\
                                &\le\lambda^2\cdot \|u\|_{H^1_\lambda}^2.
\end{split}
\end{equation}
So $[\Delta, \chi_0]$ is bounded with norm $C\lambda$ from $ H_\lambda^1(M) \to L^2(M)$.

\item\emph{Term $\tilde{\chi_2}R_M\psi_1$}
Since the cut off functions $\tilde{\chi_2}$ and $\psi_1$ are both disjoint from $K$, by Proposition~\ref{prop:Nontrappingcase}, we have $\|\tilde{\chi_2} R_M\psi_1\|_{L^2\to L^2}$ is bounded with norm $\lambda^{-1}$. 
\end{itemize}
We combine everything together:
 \begin{equation}
     \begin{split}
         \|\tilde{\chi_2}R_M(\lambda \mp i0)\chi_0^2\|_{L^2\to L^2}&=\|(\tilde{\chi_2} R_M\psi_1 [\Delta,\chi_0]\psi_1 R_{\tilde{ M}}\chi_0)(\lambda\mp i0)\|_{L^2\to L^2}\\
         &\leqslant  C \lambda^{-1}\sqrt{\log \lambda}\cdot\lambda\cdot \lambda^{-1}\\
         &= C\lambda^{-1}\sqrt{\log \lambda}.
     \end{split}
 \end{equation}
Then, by taking the adjoint again, we have the estimate of $\chi_0^2R_M(\lambda \pm i0)\tilde{\chi_2}$.

Now we consider the general case $\chi_LR_M\chi_R$. We assume $\operatorname{supp}\chi_L\cap K\neq\emptyset$ and $\operatorname{supp}\chi_R\cap K=\emptyset$, then we can consider the previous $\chi_0$: 
\[
\chi_LR_M\chi_R=\chi_L(\chi_0^2+1-\chi_0^2)R_M\chi_R=\chi_L\chi_0^2R_M\chi_R+\chi_L(1-\chi_0^2)R_M\chi_R.
\]
The first term of the right-hand side:
\[
\|\chi_L\chi_0^2R_M(\lambda\pm i0)\chi_R\|_{L^2\to L^2}\le \|\chi_0^2R_M(\lambda \pm i0)\chi_R\|_{L^2\to L^2}\lesssim \lambda^{-1}\sqrt{\log \lambda}.
\]
The second term is controlled by Proposition~\ref{prop:Nontrappingcase}.
\end{proof}
\begin{prop}[Trapped case]\label{prop:Trappedcase}
    Suppose $\chi_L$ and $\chi_R$ are two cutoff functions, 
    Then we have 
    \[
    \|\chi_LR_{M}(\lambda\pm i0)\chi_R\|_{L^2\to L^2}\lesssim \lambda^{-1}\log \lambda.
    \]
\end{prop}
\begin{proof}
We also first consider a special case:\ $\chi_0^2R_M\chi_0$. By taking $\kappa=\chi_0$ in Lemma~\ref{lem:comparision}, we have
\begin{equation}
    \chi_0^2 R_M(\lambda \pm i0)\chi_0=\chi_0 R_{\tilde M}(\lambda \pm i0)\chi_0^2+(\chi_0R_{\tilde M}\psi[-\Delta_M,\chi_0]\psi R_{M}\chi_0)(\lambda\pm i0).
\end{equation}
By taking adjoint, we have 
\[
\chi_0R_M(\lambda\mp i0)\chi_0^2=\chi_0^2 R_{\tilde M}(\lambda \mp i0)\chi_0-(\chi_0R_{M}\psi[-\Delta_M,\chi_0]\psi R_{\tilde M}\chi_0)(\lambda\mp i0).
\]
The first term $\chi_0^2 R_{\tilde M}\chi_0$ is bounded by $\lambda^{-1}\log \lambda$ from Lemma~\ref{lem:XZ25-resolvent}. The second term can be estimated similarly to Proposition~\ref{prop:Half-trappedcase}, which is bounded by $\lambda^{-1}\log \lambda$. Then, by taking the adjoint again, we have the estimate of $\chi_0^2 R_M(\lambda \pm i0)\chi_0$.

Now, for general $\chi_LR_M\chi_R$, we have split it into four parts:
\begin{equation}\label{thm1:finalestimate}
\begin{split}
    \chi_LR_M\chi_R=&\chi_L(\chi_0+1-\chi_0)R_M(\chi_0^2+1-\chi_0^2)\chi_R=\chi_L\chi_0R_M\chi_0^2\chi_R\\
    &+\chi_L(1-\chi_0)R_M\chi_0^2\chi_R+\chi_L\chi_0R_M(1-\chi_0^2)\chi_R\\
    &+\chi_L(1-\chi_0)R_M(1-\chi_0^2)\chi_R
\end{split}
\end{equation}
\begin{itemize}
    \item Then the first term of the right-hand side:
    \[
    \|\chi_L\chi_0R_M\chi_0^2\chi_R f\|_{L^2}\le \|\chi_0R_M\chi_0^2\chi_Rf\|_{L^2}\lesssim \lambda^{-1}\log \lambda \|\chi_R f\|_{L^2}\le \lambda^{-1}\log \lambda \| f\|_{L^2},
    \]
    hence is bounded by $\lambda^{-1}\log\lambda$.
    \item The second and the third term, by Proposition~\ref{prop:Half-trappedcase}, is bounded by $\lambda^{-1}\sqrt{\log\lambda}$.
    \item The last term, by Proposition~\ref{prop:Nontrappingcase}, is bounded by $\lambda^{-1}$.
\end{itemize}
This finishes the proof.
\end{proof}

Combine Proposition~\ref{prop:Nontrappingcase}, \ref{prop:Half-trappedcase}, and \ref{prop:Trappedcase}, we then completes the proof of Theorem~\ref{Thm}.
\end{proof}

\begin{proof}[Proof of Proposition~\ref{prop:propagation}]
Recall that
\begin{equation*}
     u=(-\Delta_g-(n-1)^2/4-(\lambda\pm i0)^2)^{-1}\chi f.
\end{equation*}
Using semiclassical parameter $h=\lambda^{-1}$, we have
\begin{equation*}
    P_hu=(-h^2\Delta_g-h^2(n-1)^2/4-1) u = h^2\chi f.
\end{equation*}
\emph{Step 1. Easy term}
Choose $A_\Sigma\in\Psi^{\mathrm{comp}}_h$ whose symbol equals to $1$ microlocally on a neighborhood of $\Sigma\cap T^*_{\operatorname{supp}\chi_1}M$ and with $\operatorname{WF}_h(\chi_L(1-A_\Sigma))\cap\Sigma=\emptyset$. On
$\operatorname{WF}_h(\chi_L(1-A_\Sigma))$ we have $p\neq0$. Choose a $B\in \Psi^{\mathrm{comp}}_h$, which is elliptic on $\operatorname{WF}_h(\chi_L(1-A_\Sigma))$, then by elliptic estimate(Lemma~\ref{lem:elliptic}):
we have
\begin{equation}\label{eq:elliptic}
\|\chi_L(1-A_\Sigma)u\|_{L^2}\lesssim \|Bh^2\chi f\|_{H^{-2}_h}+h^{N}\|\chi_e u\|_{H^{-N}_h}\lesssim \lambda^{-2}\|\chi f\|_{L^2}+h^N\|\chi_eu\|_{L^2}.
\end{equation}
It therefore suffices to estimate $A_\Sigma u$, i.e. the part of $\chi_Lu$ microlocalised near $\Sigma$.

\emph{Step 2. Partition in the phase space}

We will consider the geodesic flow, i.e.,\ $\varphi_{t}$. For the wavefront set of $A_\Sigma$ denoted as $\Lambda$.
For a small neighborhood of $K$ call $W$, and a operator $\beta_0\in\Psi^{\mathrm{comp}}_h$ and $\beta_0$ is elliptic on $K$. Then we split $A_\Sigma$:
\begin{equation}
    A_\Sigma=\sum_{k}A_k+O(h^\infty), \,A_k\in \Psi^{\mathrm{comp}}_h.
\end{equation}
With $\operatorname{WF}_h(A_k)\subset W_k$, where $W_k$ is a cover of $\Lambda$. For each $W_k$ lies in one of two cases:
\begin{itemize}
    \item $W_k\cap \Gamma_+=\emptyset$, we call this the backward escaping case.
    \item $W_k\cap \Gamma_-=\emptyset$, we call this the backward escaping case.
\end{itemize}
By the finiteness of the cover, we may assume that there exists a uniform constant $T$, such that $\varphi_{-t}(W_k)$ is contained in $\{\chi_2=1\}$ or $\varphi_t(W_k)\subset \{\chi_2=1\}$ for $t>T$. Hence, we can split the cover $\{W_k\}$ into two families of sets by forward escaping and backward escaping, namely,
\[
\{W_k\}_{k=1}^n=\{W_k\}_{k=1}^l\cup\{W_k\}_{l+1}^n,
\]
where the first $l$ sets are forward escaping, and the rest are backward escaping.
We choose $\tilde{\chi}\in \Psi^{\mathrm{comp}}_h$ whose symbol is $1$ on a neighborhood of $\Lambda$ and all flowing area $W_k$: \ $\{\varphi_{-t}(W_k), 0<t<T\}$ or $\{\varphi_{t}(W_k), 0<t<T\}$

\emph{Step 3. Propagation estimate}

Since $ P_h u=h^2\chi f$, we have
\[
h^{-1}\|B_1 P_h u\|_{H_h^{-1}}\lesssim h\|\chi f\|_{L^2}=\lambda^{-1}\|\chi f\|_{L^2}.
\]
Define 
\[
\chi_{\infty,1}=1 \text{ on }\varphi_T(\{W_k\}_{1}^l)+\varphi_{-T}(\{W_k\}_{l}^n) \text{ and }\operatorname{supp}\chi_{\infty,1}\subset \{\chi_2=1\}.
\]
Apply \eqref{eq:propagation_estimate} to each $A_k$,
\begin{itemize}
    \item For forward escaping $W_k$, set $B=\chi_{\infty,1}$, we have
    \begin{equation}
        \|A_k u\|_{L^2}\le C\|\chi_{\infty,1} u\|_{L^2}+C\lambda^{-1}\|\chi f\|_{L^2}+O(h^\infty)\|\tilde\chi u\|_{L^2} .
    \end{equation}
    \item For backward escaping $W_k$, also set $B=\chi_{\infty,1}$, we have the same inequality.
\end{itemize}
We sum them and the easy case in Step 1 together, then we have
\begin{equation}\label{eq:raw}
\|\chi_Lu\|_{L^2}\lesssim \|\chi_{\infty,2} u\|_{L^2}+\lambda^{-1}\|\chi f\|_{L^2}+O(\lambda^{-N})\|\tilde\chi u\|_{L^2} ,
\end{equation}

\emph{Step 4. Absorption}

We will estimate the last term of \eqref{eq:raw}. Fix once and for all:
\begin{itemize}
\item $\psi\in C_c^\infty(M)$ with $\psi=1$ on $\operatorname{supp}\nabla\chi_0$ and $\operatorname{supp}\psi\cap K=\emptyset$;
\item $\psi_1\in C_c^\infty(M)$ with $\psi_1=1$ on $\operatorname{supp}\nabla\psi\cup\operatorname{supp}\psi$ and $\operatorname{supp}\psi_1\cap K=\emptyset$;
\item $\chi_c\in C_c^\infty(U_C)$ with $\chi_c=1$ on $\operatorname{supp}\chi_0\cup\operatorname{supp}\psi$.
\end{itemize}
Before we go into the estimate, we list two lemmas:
\begin{lem}[$H^1_\lambda$ control away from the source]\label{appendix:h1}
For $\lambda\ge C$,
\begin{equation}\label{eq:h1}
\|{\psi u}\|_{H^1_\lambda}\lesssim\|{\psi_1 u}\|_{L^2}+\lambda^{-1}\|{\chi f}\|_{L^2}.
\end{equation}
\end{lem}
\begin{proof}
Pair $(-\Delta-(n-1)^2/4-(\lambda\pm i0)^2)u=\chi f$ with $\psi^2 u$ and take real parts. Using the standard identity
$\langle -\Delta_g u,\psi^2u\rangle=\|{\nabla(\psi u)}\|^2_{L^2}-\|{u\,\nabla\psi}\|^2_{L^2}$
, we get
\[
\|{\nabla(\psi u)}\|_{L^2}^2\le\|{u\,\nabla\psi}\|_{L^2}^2+\big(\lambda^2+\frac{(n-1)^2}{4}\big)\|{\psi u}\|_{L^2}^2+\|{\chi f}\|_{L^2}\|{\psi^2u}\|_{L^2}.
\]
Since $|\nabla\psi|\lesssim\psi_1$ and $\psi\le\psi_1$, dividing by $\lambda^2$ and using $ab\le a^2+b^2$ on the last term gives
\[
\lambda^{-2}\|{\nabla(\psi u)}\|_{L^2}^2\lesssim\|{\psi_1u}\|_{L^2}^2+\lambda^{-2}\|{\chi f}\|_{L^2}^2,
\]
which together with $\|{\psi u}\|_{L^2}\le\|{\psi_1 u}\|_{L^2}$ is \eqref{eq:h1}, recalling the definition \eqref{eq:sobolev} of $H^1_\lambda$.
\end{proof}
The second lemma is:
\begin{lem}\label{appendix:apriori}
For $\lambda\ge C$,
\begin{equation}\label{eq:apriori}
\|{\chi_0 u}\|_{L^2}\lesssim\log\lambda(\|{\psi_1 u}\|_{L^2}+\lambda^{-1}\|{\chi f}\|_{L^2}).
\end{equation}
\end{lem}
\begin{proof}
On $\operatorname{supp}\chi_0\subset U_C$ we have $g=\tilde g$, hence
\[
\big(-\Delta_{\tilde g}-\frac{(n-1)^2}4-(\lambda\pm i0)^2\big)(\chi_0u)
=\chi_0\chi f+[-\Delta_g,\chi_0]\,u
=\chi_0\chi f+\psi[-\Delta_g,\chi_0]\psi\,u ,
\]
where we used $\psi=1$ on $\operatorname{supp}\nabla\chi_0$. So we have
\[
\chi_0u = R_{\tilde{M}}(\lambda)\left[\chi_0\chi f+\psi[-\Delta_g,\chi_0]\psi u\right].
\]
The source is supported in $\{\chi_c=1\}$, so $\chi_0u=\chi_cR_{\tilde{M}}(\lambda)\chi_c\left[\cdots\right]$ and by Lemma~\ref{lem:XZ25-resolvent} applies:
\begin{equation}
\begin{split}
    \|{\chi_0u}\|_{L^2}\le\|{\chi_cR_{\tilde{M}}(\lambda)\chi_c}\|_{L^2\to L^2}\left(\|{\chi f}\|_{L^2}+\|{[-\Delta_g,\chi_0]\psi u}\|_{L^2}\right)\\
\lesssim\lambda^{-1}\log\lambda\left(\|{\chi f}\|_{L^2}+\|{[-\Delta_g,\chi_0]\psi u}\|_{L^2}\right).
\end{split}  
\end{equation}

It's can be shown easily, the first-order operator $[-\Delta_g,\chi_0]$ is bounded $H^1_\lambda\to L^2$ with norm $C\lambda$, so by Lemma~\ref{appendix:h1}
\[
\|{[-\Delta_g,\chi_0]\psi u}\|_{L^2}\lesssim\lambda\|{\psi u}\|_{H^1_\lambda}\lesssim\lambda\|{\psi_1u}\|_{L^2}+\|{\chi f}\|_{L^2}.
\]
Combining the two estimates yields \eqref{eq:apriori}.
\end{proof}

For the last term of \eqref{eq:raw}, by \eqref{eq:p.o.u}, we have 
\[
\|\tilde{\chi} u\|_{L^2}=\|(\eta_0+\eta_1+\eta_2)\tilde{\chi}u\|_{L^2}\le \|\eta_1 u\|_{L^2}+\|\eta_2\tilde{\chi} u\|_{L^2}+\|\chi_0 u\|_{L^2}.
\]
Then by Lemma~\ref{appendix:apriori}, we have
\begin{equation*}
    \|\tilde{\chi} u\|_{L^2}\lesssim \|\eta_1 u\|_{L^2}+\|\chi_2\tilde{\chi} u\|_{L^2}+\log \lambda \| \psi_1 u\|_{L^2}+\lambda^{-1}\log \lambda \|\chi f\|_{L^2}.
\end{equation*}
Therefore in \eqref{eq:raw}, we have
\begin{equation}
    \|\chi_Lu\|_{L^2}\lesssim \|\chi_{\infty,2}u\|_{L^2}+\lambda^{-1}\|\chi f\|_{L^2}+O(\lambda^{-N})\left(\|\psi_1 u\|_{L^2}+\|\eta_1 u\|_{L^2}+\|\chi_2\tilde{\chi} u\|_{L^2}\right).
\end{equation}
Since $\operatorname{supp}\eta_2\subset\{\chi_2=1\}$, we can absorb $\eta_2\tilde{\chi}$ into $\chi_{\infty,2}$. After doing this, we have
\begin{equation}
    \|\chi_Lu\|_{L^2}\lesssim \|\chi_{\infty,2}u\|_{L^2}+\lambda^{-1}\|\chi f\|_{L^2}+O(\lambda^{-N})\left(\|\psi_1 u\|_{L^2}+\|\eta_1 u\|_{L^2}\right).
\end{equation}
Define $G(u):=\sup\{\|\psi_1u\|_{L^2},\|\eta_1 u\|_{L^2}\}$. Then we have
\begin{equation}
    G(u)\lesssim\|\chi_{\infty,3} u\|_{L^2}+\lambda^{-1}\|\chi f\|_{L^2}+ O(\lambda^{-N} G(u)).
\end{equation}
Hence 
\begin{equation}\label{eq:lowerorderterm}
    G(u)\lesssim \|\chi_{\infty,3} u\|_{L^2}+\lambda^{-1}\|\chi f\|_{L^2}.
\end{equation}
Finally, let $\chi_{\infty}\in C_c^{\infty}(M_a)$ such that $\chi_{\infty}(x)=1$ on $\supp\chi_{\infty,2}\cup \supp\chi_{\infty,3}$. By \eqref{eq:raw} and \eqref{eq:lowerorderterm},
\begin{equation}
    \|\chi_L u\|_{L^2}\lesssim \|\chi_\infty u\|_{L^2}+\lambda^{-1}\|\chi f\|_{L^2}.
\end{equation}
This finishes the proof.
\end{proof}

\begin{proof}[Proof of Corollary~\ref{cor:main}]
    The corollary follows from Theorem~\ref{Thm}, where the assumption (i) follows from Proposition~\ref{ass:CV}, the assumption (ii) follows from Proposition~\ref{prop:glue}, and the assumption (iii) follows from \cite{tao2021spectral}.
\end{proof}
    
\section{The comparison manifold and its spectral properties}\label{sec:comparison}

This section carries out the only new ingredient in the proof of Theorem~\ref{thm2}, which is geometric. To each end of $M$ we attach a model hyperbolic disc, producing a surface $(N,g_N)$ that is globally close to a hyperbolic disc, has strictly negative curvature, and has no $L^2$-eigenvalue and no resonance at the bottom of its continuous spectrum. The harmonic analysis machinery that turns these spectral properties into the Strichartz and spectral-projection estimates is borrowed from \cite{Bounded_HSTZ}. The proof of Theorem~\ref{thm2} is then given in Section~\ref{sec:thm2}.

We are looking for a two-manifold $(N,g_N)$ with the following properties:
\begin{enumerate}
\item[\rm(a)] $N$ is a finite disjoint union of simply connected asymptotically hyperbolic surfaces;
\item[\rm(b)] $g_N$ has strictly negative curvature;
\item[\rm(c)] $g_N$ agrees with $g$ on a neighborhood of infinity;
\item[\rm(d)] $-\Delta_{g_N}$ has no $L^2$-eigenvalue and no resonance at the bottom of the continuous spectrum.
\end{enumerate}
In \cite{Bounded_HSTZ} this surface is denoted $(\tilde M,\tilde g)$, here we reserve $(\tilde M,\tilde g)$ for the even AH surface of the previous sections and write $(N,g_N)$ instead, to avoid any clash of notation.

\subsection{Construction of the model manifold}\label{subsec:Nconstruction}

We use the hyperbolic disc as a model. Recall that for the hyperbolic disc $\bbH^2$,
\[
g_{\mathbb H^2}=dr^2+\sinh^2(r)\,d\theta^2,\quad
(r,\theta)\in[0,\infty)\times S^1\big/\sim,
\]
and, after the change of variable $y=\theta$, $x=e^{-r}$,
\begin{equation}\label{eq:hypfunnel}
g_{\mathbb H^2}=\frac{dx^2+h_0(x)\,dy^2}{x^2},\quad
h_0(x)=\frac{(x^2-1)^2}{4},\quad x\in(0,1].
\end{equation}

\paragraph{Filling in each end of $M$.}
Since we always work on each component separately, we may assume $\partial M=S^1$. Fix a cutoff $\chi\in C_c^\infty((-1,1))$ with $\chi=1$ on $(-1/2,1/2)$, and a large constant $R\gg 1$ to be chosen later. For the end $E\cong(0,\varepsilon)\times S^1\subset M$, define the model disk
\[
N= \bigl((0,1]\times S^1\bigr)\big/\sim, \quad
\sim\colon\{x=1\}\times S^1 \mapsto \text{pt},
\]
and equip it with the metric
\begin{equation}\label{eq:gN}
g_N = \frac{dx^2 + \tilde h(x,y)\,dy^2}{x^2},\quad
\tilde h(x,y)=h_0(x)+\chi(Rx)\cdot\bigl[h(x,y)-h_0(x)\bigr],
\end{equation}
where $h(x,y)$ is the function in the definition of an asymptotically hyperbolic manifold (Def~\ref{def:AH}). Then $h(x,y)$ can be chosen after changing the coordinate satisfy $h(0,y)=1/4$.
Then $g_N$ agrees with $g$ on $\{x\leqslant 1/(2R)\}$, while for $x\ge 1/R$ we have $\chi(Rx)=0$ and $\tilde h=h_0$, so $g_N=g_{\mathbb H^2}$.

For the curvature bound, note that on $\operatorname{supp}\chi(R\cdot)$ we have $x\leqslant R^{-1}$. Since $h$ and $h_0$ are both $C^1$ in $x$ and $h(0,y)=h_0(0)=1/4$, Taylor's theorem gives
\begin{equation}\label{eq:estimate_h}
  |h(x,y)-h_0(x)|\lesssim x\lesssim R^{-1}
\quad\text{on}\ \operatorname{supp}\chi(R\cdot).  
\end{equation}

Since we will need $(N,g_N)$ to be strictly negatively curved, we first give the following lemma:
\begin{lem}
    For $R$ large enough, the Gauss curvature $K_N$ of manifold $(N,g_N)$ satisfies 
    \[
    -\frac{3}{2}\le K_N\le -\frac{1}{2}.
    \]
\end{lem}
\begin{proof}
    The computation is similar to Proposition~\ref{prop:glue}. For metric 
    \[
    g=\frac{dx^2 + a(x,y)^2\,dy^2}{x^2},
    \]
    the expression of Gauss curvature is 
    \[
    K_g=-1+\frac{xa_x}{a}-\frac{x^2 a_{xx}}{a}.
    \]
    In our case, $\tilde h(x,y)=h_0(x)+\chi(Rx)\cdot\bigl[h(x,y)-h_0(x)]$. We denote $\rho(x,y)=h(x,y)-h_0(x)$, then $\rho(0,y)=0$ implies $\rho(x,y)=x\rho_1(x,y)$. Then for $R(x,y)=\chi(Rx)x\rho_1(x,y)$, taking derivative with respect to $x$, we have
    \begin{align*}
R_x   &= \chi \rho_1 + Rx \chi' \rho_1 + Rx \chi \rho_{1,x}, \\
R_{xx} &= 2R\chi'\rho_1 + 2\chi\,\rho_{1,x}
        + R^2x\,\chi''\rho_1 + 2Rx\chi'\rho_{1,x} + x\,\chi\,\rho_{1,xx}.
\end{align*}
Hence the Gauss curvature of $(N,g_N)$ has the expression:
\[
K_{g_N}=-1+\frac{x R_x}{\sqrt{\tilde h}}-\frac{x^2R_{xx}}{\sqrt{\tilde{h}}}.
\]
Since we will only focus on the area $\{\chi(Rx)\neq 1,0\}$, which is near the boundary, so we can assume $\tilde{h}(x,y)>\varepsilon_0>0$. And by the computation above, we have
\[
|R| = O(R^{-1}),\quad
|R_x| = O(1),\quad
|R_{xx}| = O(R).
\]
After multiplying with $x$ or $x^2$, we have $|xR_x|$ and $|x^2 R_{xx}|$ is of $O(R^{-1})$. So after choose $R$ large, we have $K_{g_N}$ is closed to $-1$, which completes the proof.
\end{proof}
$(N,g_N)$ has strictly negative curvature, so (a)--(c) hold. (d) will be proved in Lemma~\ref{on_eigen}.

\subsection{Absence of eigenvalues and resonances}\label{subsec:noresonance}

Recall that an eigenvalue of $-\Delta_{g_N}$ is a number $\mu\in\mathbb{R}$ for which there exists $u\in L^2$, $u\neq 0$, with $-\Delta_{g_N}u=\mu u$. Resonance is defined as the poles of the meromorphic extension of the resolvent $R_M(\lambda)$ (see \cite{dyatlov2019mathematical} and \cite{guillarmou2005meromorphic} for the meromorphic extension in the non-even case).

\begin{lem}\label{on_eigen}
Suppose $R$ is sufficiently large. Then the Laplacian $-\Delta_{g_N}$ for the metric $g_N$ on $N$ has no eigenvalue in $[0,1/4]$ and has no resonance at $1/4$, equivalently at $\lambda=0$.
\end{lem}

\begin{proof}

\emph{Step 1.} By a direct computation,
\begin{equation}
\Delta_{g} f = x^2 \partial_x^2 f
+ x^2 \left(\frac{\partial_x h}{2h}\right) \partial_x f
+ \frac{x^2}{h} \partial_y^2 f
- x^2 \left(\frac{\partial_y h}{2h^2}\right) \partial_y f.
\end{equation}
Hence
\begin{align}
\Delta_{g_{\mathbb{H}^2}} f 
&= x^2 \partial_x^2 f 
   + x^2\,\frac{h_0'(x)}{2h_0(x)}\,\partial_x f 
   + \frac{x^2}{h_0(x)}\,\partial_y^2 f,\\
\Delta_{g_N} f 
&= x^2 \partial_x^2 f 
   + x^2\,\frac{\partial_x \tilde{h}(x,y)}{2\tilde{h}(x,y)}\,\partial_x f 
   + \frac{x^2}{\tilde{h}(x,y)}\,\partial_y^2 f 
   - x^2\,\frac{\partial_y \tilde{h}(x,y)}{2\tilde{h}(x,y)^2}\,\partial_y f,
\end{align}
so that
\begin{equation}\label{cha}
\Delta_{g_N} - \Delta_{g_{\mathbb{H}^2}}= x^2 \left( \frac{\partial_x \tilde{h}}{2\tilde{h}} - \frac{\partial_x h_0}{2h_0} \right) \partial_x + x^2 \left( \frac{1}{\tilde{h}} - \frac{1}{h_0} \right) \partial_y^2
   - x^2 \frac{\partial_y \tilde{h}}{2\tilde{h}^2} \partial_y.
\end{equation}
On $\operatorname{supp}\chi(R\cdot)$ we have $x\leqslant 1/R$ and $\tilde{h}(0,y)=h_0(0)=1/4$, whence 
\[
\frac{1}{\tilde{h}}-\frac{1}{h_0}\lesssim x\lesssim R^{-1},\quad \partial_x\tilde{h}-\partial_x h_0=R\,\chi'(Rx)\,(h(x,y)-h_0(x))+\chi(Rx)\,O(1).
\]
For the first term on the right-hand side of the second equation, $\chi'\neq0$ forces $Rx\sim 1$, hence $x\sim R^{-1}$. Combining this with \eqref{eq:estimate_h}, so in all $\partial_x\tilde{h}-\partial_x h_0=O(1)$.

Comparing the resolvents of $g_N$ and $g_{\bbH^2}$,
\[
R_{g_N}(\lambda)-R_{g_{\bbH^2}}(\lambda)=-R_{g_{\bbH^2}}(\lambda)\circ (\Delta_{g_N}-\Delta_{g_{\bbH^2}})\circ R_{g_N}(\lambda),
\]
hence
\begin{equation}\label{idd}
R_{g_{\bbH^2}}(\lambda)=(1+K(\lambda))R_{g_N}(\lambda),\quad K(\lambda)=R_{g_{\bbH^2}}(\lambda)\circ (\Delta_{g_N}-\Delta_{g_{\bbH^2}}).
\end{equation}
From \cite{borthwick2016spectral}, we know that $R_{g_{\bbH^2}}(\lambda)$ can be extended to the whole complex plane, and for any $a>0$, the extension can be viewed as(in \cite{guillarmou2005meromorphic}):
\[
(\operatorname{Im}\lambda>0)\to \mathcal{L}(L^2,L^2)\quad\Longrightarrow\quad  (\operatorname{Im}\lambda>-a)\to \mathcal{L}(x^aL^2,x^{-a}L^2);
\]
as a meromorphic map. Since $\lambda=0$ is not a resonance of $R_{g_{\bbH^2}}$(see \cite{borthwick2016spectral} for explicite computation), it is holomorphic around $\lambda=0$. 

\emph{Step 2.} In the rest of the proof, $L^2$ and $H^2$ are taken with respect to $g_{\bbH^2}$. Suppose $u=R_{g_{\bbH^2}}(0)f$ with $f\in x^{a}L^2$; then $u\in x^{-a}L^2$, and $v=x^a u\in L^2$. From $(\Delta_{\bbH^2}+1/4)u=f$ we obtain
\[
x^a(\Delta_{g_{\bbH^2}}+1/4)x^{-a}v=x^a f.
\]
The operator $x^a(\Delta_{g_{\bbH^2}}+1/4)x^{-a}$ has the same principal part as $\Delta_{g_{\bbH^2}}+1/4$:
\begin{equation}
x^a(\Delta_{g_{\bbH^2}})x^{-a} v=x^2\partial_x^2v+\frac{x^2}{h_0(x)}\partial_y^2v+b(x)x\partial_xv+ c(x)v,
\end{equation}
with 
\[
b(x)=x\frac{h_0'(x)}{2h_0(x)}-2a,\quad c(x)=a(a+1)-ax\frac{h_0'(x)}{2h_0(x)},
\]
smooth and bounded. Since $\bbH^2$ is complete with bounded geometry, the elliptic regularity of $x^a(\Delta_{g_{\bbH^2}}+1/4)x^{-a}$ gives
\begin{equation}\label{es:elliptic}
\|v\|_{H^2}\lesssim \|x^af\|_{L^2}+\|v\|_{L^2}.
\end{equation}
As $\|x^af\|_{L^2},\,\|v\|_{L^2}<\infty$, this implies $v\in H^2$. Hence $u\in x^{-a}H^2$, and $R_{g_{\bbH^2}}(1/2)\colon x^aL^2\to x^{-a}H^2$.

\emph{Step 3.} Consider the operator $V=\Delta_{g_N}-\Delta_{g_{\bbH^2}}$ on $x^{-a}H^2$. For $u\in x^{-a}H^2$ we estimate $\|x^{-a}Vu\|_{L^2}$. By \eqref{cha} we may write
\[
V=a_1(x,y)(x\partial_x)+a_2(x,y)(x\partial_y)^2+a_3(x,y)(x\partial_y),
\]
and from $\partial_y\tilde{h}\lesssim x$,
\[
|a_1(x,y)|\lesssim x, \quad|a_2(x,y)|\lesssim x,\quad |a_3(x,y)|\lesssim x^2,\quad \operatorname{supp}a_j\subset\{|x|\leqslant R^{-1}\}.
\]
Then
\begin{equation}
\begin{split}
\|x^{-a}Vu\|_{L^2}^2
&=\int_{x\leqslant R^{-1/2}}|Vu|^2\,x^{-2a}\,dV_{g_{\bbH^2}}\\
&\lesssim\int_{x\leqslant R^{-1/2}} \left(x^2|x\partial_x u|^2+x^2|(x\partial_y)^2u|^2+x^4|x\partial_y u|^2\right)\,x^{-2a}\,dV_{g_{\bbH^2}}\\
&\lesssim\int_{x\le R^{-1/2}}
   \left(|x\partial_x u|^2+|(x\partial_y)^2u|^2+|x\partial_y u|^2\right)\,x^{2-2a}\,dV_{g_{\bbH^2}}.
\end{split}
\end{equation}
Choosing $a<1/2$, we have $x^{2-2a}=x^{2-4a}\cdot x^{2a}$ with $2-4a>0$, hence $x^{2-4a}\leqslant R^{-1+2a}$ on $\{x\leqslant R^{-1/2}\}$, so
\begin{equation}\label{es:V}
\begin{split}
\|x^{-a}Vu\|_{L^2}^2
&\lesssim R^{2a-1}\int_{x\leqslant R^{-1/2}}\left(|x\partial_x u|^2+|(x\partial_y)^2u|^2+|x\partial_y u|^2\right)\,x^{2a}\,dV_{g_{\bbH^2}}\\
&\lesssim R^{2a-1}\cdot\|x^au\|_{H^2}^2.
\end{split}
\end{equation}
Thus
\[
V=\Delta_{g_N}-\Delta_{g_{\bbH^2}}\colon x^{-a}H^2\to x^aL^2\quad\text{is bounded with norm } CR^{a-1/2}\to 0\ (R\to \infty).
\]
Consequently
\[
K(\lambda)\colon \,x^{-a}H^2\xrightarrow{V} \, x^aL^2\xrightarrow{R_{g_{\bbH^2}}(\lambda)}\,x^{-a}H^2
\]
is bounded for $\lambda$ near $1/2$, and by the above its norm is $<1$ for $R$ large enough. Hence the Neumann series
\[
(1+K(\lambda))^{-1}=\sum_{j\geqslant0} (-K(\lambda))^j
\]
is well defined and holomorphic. Finally, by \eqref{idd},
\[
R_{g_N}(\lambda)=(1+K(\lambda))^{-1}\circ R_{g_{\bbH^2}}(\lambda)\colon x^aL^2\to x^{-a}H^2
\]
is holomorphic, so $1/2$ is not a resonance of $R_{g_N}$.

\emph{Step 4.} The eigenvalue in $(0,1/4]$ coincide with the resonance in $\{\operatorname{Re}(\lambda)=0,\operatorname{Im}(\lambda)\in (0,1/2].\}$. The absence of an eigenvalue at $1/4$ is shown in \cite{bouclet2013absence}. We will focus on $\lambda\in (0,i/2]$. Since \eqref{es:elliptic} and \eqref{es:V} is always true, and the bound is independent from $\lambda$, so $R_{g_{\bbH^2}}(\lambda)\colon x^aL^2\to x^{-a}H^2$ and $V\colon x^{-a}H^2\to x^aL^2$ uniformly bounded for $\lambda\in [0,i/2]$. So if we choose $R$ sufficiently large, then $K(\lambda)$ will be invertible for $\lambda\in [0,i/2]$. In the end,
\[
R_{g_N}(\lambda)=(1+K(\lambda))^{-1}\circ R_{g_{\bbH^2}}(\lambda)\colon x^aL^2\to x^{-a}L^2,
\]
is holomorphic for $\lambda\in [0,i/2]$, hence there is no eigenvalue in $[0,1/4]$.
\end{proof}

\section{Proof of Theorem~\ref{thm2}}\label{sec:thm2}

Throughout this section, $(M,g)$ is a smooth two-dimensional asymptotically hyperbolic surface with sectional curvature $\le -k^2<0$. The even case is treated in \cite{Bounded_HSTZ}; we explain the changes needed in the non-even setting. We first record the estimates that enter the argument, then prove the global Strichartz estimate in detail, and finally indicate the spectral-projection estimate.

\subsection{Estimates used as black boxes}\label{subsec:inputs}

We collect the estimates that enter the proof. Throughout, $P=\sqrt{-\Delta_g}$ and $\pi(K)$ denotes the projection of the trapped set.

\medskip
\noindent\textbf{(1) Non-trapping resolvent estimate.} For $\chi\in C_c^\infty(M)$ supported away from $\pi(K)$, Proposition~\ref{prop:Nontrappingcase} gives
\begin{equation}\label{far_away}
\|\chi\,R_M(\lambda\pm i0)\,\chi\|_{L^2\to L^2}\leqslant C\lambda^{-1},
\qquad \lambda>C.
\end{equation}

\noindent\textbf{(2) Log-loss resolvent estimate.} By Theorem~\ref{Thm},
\begin{equation}\label{any}
\|\chi\,R_M(\lambda\pm i0)\,\chi\|_{L^2\to L^2}\leqslant C\lambda^{-1}\log\lambda,
\qquad \chi\in C_c^\infty(M).
\end{equation}

\noindent\textbf{(3) Local smoothing.} From \eqref{far_away}--\eqref{any} and \cite[Theorem~7.2]{dyatlov2019mathematical}, for $\beta\in C_c^\infty((1/2,2))$,
\begin{align}
\|\chi\,e^{-it\Delta_g}\beta(P/\lambda)u_0\|_{L^2_{t,x}(M\times\mathbb R)}
&\leqslant C\lambda^{-1/2}(\log\lambda)^{1/2}\|u_0\|_{L^2},
\qquad \chi\in C_c^\infty(M),\label{lsLL}\\
\|\chi_1\,e^{-it\Delta_g}\beta(P/\lambda)u_0\|_{L^2_{t,x}(M\times\mathbb R)}
&\leqslant C\lambda^{-1/2}\|u_0\|_{L^2},
\qquad \chi_1\in C_c^\infty(M\setminus\pi(K)).\label{lsNT}
\end{align}

\noindent\textbf{(4) Non-trapping Strichartz.} 
For the manifold $(N,g_N)$ in Section \ref{sec:comparison}, by Lemma~\ref{on_eigen} and \cite[Theorem 1]{chen2018resolvent}, for $\frac{1}{p}+\frac{1}{q}=\frac{1}{2}, \,p,q\ge 2$ and $(p,q)\neq(2,+\infty)$, we have 
\begin{equation}\label{Non-trapping}
\|e^{-it\Delta_{g_N}}u_0\|_{L^p_tL^q_x(N\times \mathbb{R})}\le C\|u_0\|_{L^2}.
\end{equation}

\noindent\textbf{(5) Log-scale lossless Strichartz.} Since $(M,g)$ has uniformly bounded geometry and nonpositive sectional curvature, \cite[Thm 1.4]{Bounded_HSTZ} gives
\begin{equation}\label{StrLog}
\|\beta(P/\lambda)e^{-it\Delta_g}u_0\|_{L^p_tL^q_x(M\times[0,\lambda^{-1}\log\lambda])}
\le C\|u_0\|_{L^2}.
\end{equation}

\noindent\textbf{(6) Log-scale sharp spectral projection.} Since the sectional curvatures of $M$ are below $-k^2$ and $M$ has bounded geometry, \cite[Thm 1.6]{Bounded_HSTZ} yields
\begin{equation}\label{spLog}
\|1_{[\lambda,\lambda+(\log\lambda)^{-1}]}(P)\|_{L^2\to L^q}
\le C\lambda^{\mu(q)}(\log\lambda)^{-1/2},
\qquad q\in(2,\infty].
\end{equation}

\noindent\textbf{(7) Non-trapping spectral projection.} For manifold $(N,g_N)$ as we construct in Section~\ref{sec:comparison}, \cite[Theroem 1.6]{chen2018resolventII} says:
\begin{equation}\label{nontrapping:spectral}
    \|\mathbf{1}_{[\lambda,\lambda+\delta]}(P) f\|_{L^q(N)}\leqslant C\lambda^{\mu(q)}\delta^{1/2}\|f\|_{L^2(N)}.
\end{equation}
where 
\begin{equation}\label{mu}
        \mu(q)=\left\{\begin{aligned}
            &1/2-2/q,\quad q\geqslant 6,\\
            &\frac{1}{2}(1/2-1/q),\quad q\in (2,6].
        \end{aligned}
        \right.
\end{equation}

\medskip
Combining these, \cite[Lem 2.1]{Bounded_HSTZ} gives the lossless local Strichartz estimate.
\begin{lem}\label{lem:lossless}
The estimates \eqref{far_away}--\eqref{StrLog} imply
\begin{equation}
\|e^{-it\Delta_g}u_0\|_{L^p_tL^q_x(M\times[0,1])}\le C\|u_0\|_{L^2}.
\end{equation}
\end{lem}

\subsection{Global Strichartz estimate}\label{subsec:strichartz}

The proof follows that of \cite[Theorem 1.1]{Bounded_HSTZ}, with the comparison manifold $(\tilde M,\tilde g)$ there replaced by the surface $(N,g_N)$ of Section~\ref{sec:comparison}; the changes are only notational. 

The proof is the same as that of \cite{Bounded_HSTZ}, once all the required estimate inputs have been established. We therefore only describe the differences between the two settings.

The first difference concerns the estimate inputs. In \cite{Bounded_HSTZ}, the relevant estimates are obtained in the setting of even asymptotically hyperbolic manifolds. In our setting, the corresponding resolvent input is replaced by Theorem~\ref{Thm}, from which the required local smoothing estimates follow.

The second main difference concerns the comparison manifold near infinity, which is required to have no resonance at the bottom of the continuous spectrum. In \cite{Bounded_HSTZ}, such a comparison manifold is constructed using Vasy's method under the evenness assumption, whereas in our setting the manifold $(N,g_N)$ is constructed in Section~\ref{sec:comparison}.

\subsection{Spectral projection estimate}\label{subsec:projection}

The spectral-projection estimate of Theorem~\ref{thm2} (2) is proved exactly as in \cite{Bounded_HSTZ}; since the computations are long and identical in our notation, we only indicate the ingredients and refer to \cite{Bounded_HSTZ} for details. The proof uses
\begin{itemize}
\item the comparison surface $(N,g_N)$ of Section~\ref{sec:comparison} (denoted $(\tilde M,\tilde g)$ in \cite{Bounded_HSTZ});
\item the local smoothing estimates \eqref{lsLL}--\eqref{lsNT};
\item the Littlewood--Paley decomposition of \cite[Lemma 2.2]{Bounded_HSTZ};
\item the almost-orthogonality estimates used in the proof of the global Strichartz estimate, namely \cite[Lemma 2.5]{Bounded_HSTZ} and \cite[Lemma 2.6]{Bounded_HSTZ}.
\item Non-trapping spectral projection \eqref{nontrapping:spectral} and log-scale sharp spectral projection \eqref{spLog}.
\end{itemize}

\printbibliography
\end{document}